\documentclass{article}

\usepackage{geometry}
\usepackage{graphicx} 
\usepackage{epstopdf} 

\usepackage[colorlinks, hypertexnames=false]{hyperref}
\hypersetup{linkcolor=blue, citecolor=red} 

\usepackage{bm}
\usepackage{amsmath} 
\usepackage{amssymb} 
\usepackage{stmaryrd} 
\usepackage{MnSymbol} 
\usepackage{multirow}

\usepackage{amsthm} 
\usepackage{mathtools} 

\usepackage[dvipsnames]{xcolor}
\definecolor{red}{rgb}{0.00,0.00,0.00}
{\numberwithin{equation}{section}
\setlength{\parindent}{1em}

\newtheorem{theorem}{Theorem}[section]
\newtheorem{lemma}{Lemma}[section]
\newtheorem{remark}{Remark}[section]

\newtheorem{corollary}{Corollary}[section]

\newcommand{\normmm}[1]{{\left\vert\kern-0.25ex\left\vert
\kern-0.25ex\left\vert #1
    \right\vert\kern-0.25ex\right\vert\kern-0.25ex\right\vert}}

\geometry{left=3cm,right=3cm,top=4cm,bottom=2.5cm}

\newcommand{\Red}[1]{\textcolor{red}{#1}}

\begin{document}           

\title{Staggered DG method with small edges for \\Darcy flows in fractured porous media}
\author{Lina Zhao\footnotemark[1]\qquad
    \;Dohyun Kim\footnotemark[2]\qquad
    \;Eun-Jae Park\footnotemark[3]\qquad
    \;Eric Chung\footnotemark[4]}
\renewcommand{\thefootnote}{\fnsymbol{footnote}}
\footnotetext[1]{Department of Mathematics,The Chinese University of Hong Kong, Hong Kong SAR, China. ({lzhao@math.cuhk.edu.hk}).}
\footnotetext[2]{Department of Computational Science and Engineering, Yonsei University, Seoul 03722, Korea. ({kim92n@yonsei.ac.kr}).}
\footnotetext[3]{Department of Computational Science and Engineering, Yonsei University, Seoul 03722, Korea. ({ejpark@yonsei.ac.kr}).}
\footnotetext[4]{Department of Mathematics, The Chinese University of Hong Kong, Hong Kong SAR, China. ({tschung@math.cuhk.edu.hk}).}

\maketitle

\begin{abstract}
    In this paper, we present and analyze a staggered discontinuous Galerkin method for Darcy flows in fractured porous media on fairly general meshes.
    A staggered discontinuous Galerkin method and a standard conforming finite element method with appropriate inclusion of interface conditions are exploited for the bulk region and the fracture, respectively.
    Our current analysis weakens the usual assumption on the polygonal mesh, which can integrate more general meshes such as elements with arbitrarily small edges into our theoretical framework.
    We prove the optimal convergence estimates in $L^2$ error for all the variables by exploiting the Ritz projection.
    Importantly, our error estimates are shown to be fully robust with respect to the heterogeneity and anisotropy of the permeability coefficients.
    Several numerical experiments including meshes with small edges and anisotropic meshes are carried out to confirm the theoretical findings.
    Finally, our method is applied in the framework of unfitted mesh.
\end{abstract}


\pagestyle{myheadings} \thispagestyle{plain}
\markboth{Zhao,Kim} {SDG method for Darcy flows in fractured porous media}

\section{Introduction}

Modeling flow in fractured porous media has drawn great attention in the past decades, being fundamental for addressing many environmental and energy problems, such as water resources management, isolation of radioactive waste and ground water contamination.
Given the wide applications of fractured model in practical applications, many advances has been made in the accomplishments of designing efficient  numerical methods for fractured porous media.
In \cite{Martin05}, a mixed finite element method is developed and error estimates are also proved.
Later, a mixed finite element method on non-matching grids is considered in \cite{DAngelo12}.
In \cite{Chave18}, a hybrid-high order method is analyzed on fairly general meshes.
The error estimates proposed therein show that the method is fully robust with respect to the heterogeneity of the permeability coefficients.
In \cite{Antonietti19}, a discontinuous Galerkin approximation for flows in fractured porous media on polytopal grids is analyzed, where optimal convergence estimates in mesh-dependent energy norm are derived on fairly general meshes possibly including elements with unbounded number of faces.
In addition to the aforementioned methods, we also mention other methods that have been developed for fractured porous media, see \cite{Lipnikov08, Angot09, Sandve12, Berrone13, Fumagalli13, Benedetto14, KBrenner16, AntoniettiMF16,chen2016,chen2017, KBrenner17, Delpra17,Boon18, Formaggia18}.

Staggered discontinuous Galerkin (DG) methods are initially introduced to solve wave propagation problems \cite{EricEngquistwave06,ChungWave}.
The salient features of staggered DG method make it desirable for practical applications and the applications to various partial differential equations important for both science and engineering have been considered in \cite{ChungKimWid13, EricCiarYu13, KimChungLee13, Cheung15, LeeKim16, ChungQiu17, ChungParkLina18, LinaParkconvection19}.
Recently, staggered DG methods have been successfully designed on fairly general meshes possibly including hanging nodes for Darcy law and the Stokes equations, respectively \cite{LinaPark, LinaParkShin}.
It is further developed with essential modifications to solve coupled Stokes and Darcy problem, and Brinkman problem \cite{LinaParkcoupledSD,LinaEricLam20}.
Staggered DG methods designed therein earn many desirable features, including:
1) It can be flexibly applied to fairly general meshes with possible inclusion of hanging nodes, and the handing nodes can be simply incorporated in the construction of the method;
2) superconvergence can be obtained, which can deliver one order high convergence with proper postprocessing scheme designed;
3) local mass conservations can be preserved, which is highly appreciated in the simulation of multiphase flow.
In addition, the mass matrix is block diagonal which is desirable when explicit time stepping schemes are used;
4) no numerical flux or penalty term is needed in contrast to other DG methods.

The purpose of this paper is to develop and analyze staggered DG method for the coupled bulk-fracture model stemming from the modeling of flows in fractured porous media, allowing more general meshes such as elements with arbitrarily small edges.
The flexibility of staggered DG method in handling fairly general meshes, and the preservation of physical properties indeed make it an attractive candidate for such kind of problems.
In this paper we propose a discretization which combines a staggered DG approximation for the problem in the bulk domains with a conforming finite element approximation on the fracture.
Unlike the strategies employed in \cite{DAngelo12,Chave18}, we impose the coupling conditions by replacing all the terms with respect to the jump and average of flux by the corresponding pressure term, which can compensate for the degrees of freedom for bulk pressure across the fracture.
The existence and uniqueness of the resulting system is proved and a rigorous error analysis is carried out.
In particular, we prove the convergence estimates under weaker assumption on the polygonal mesh by exploiting some novel strategies.
Research in this direction has drawn great attention, see \cite{Beiraoda17,BrennerSung18,CaoChen18,Antonietti19,CaoChen19} for works considering general polygonal elements allowing arbitrarily small edges.
The primary difficulty arising from \textit{a priori} error estimates lies in the fact that $L^2$ error estimate for flux is coupled with energy error of fracture pressure, which will naturally lead to suboptimal convergence for $L^2$ error of flux.
To overcome this issue, we construct the Ritz projection for fracture pressure so that the term causing suboptimal convergence can vanish.
Moreover, we are able to show that the Ritz projection superconverges to numerical approximation of fracture pressure. Then without duality argument we can achieve optimal convergence for $L^2$ error of fracture pressure and bulk pressure, respectively.
It is noteworthy that our error estimates are shown to be fully robust with respect to the heterogeneity and anisotropy
of the permeability coefficients, which is desirable feature for fractured flow simulation.
The theoretical findings are verified by a series of numerical tests.
Especially, numerical tests indicate that our method is robust to anisotropy of meshes.
We emphasize that our method allows general meshes with arbitrarily small edges, thus it can be easily adapted to solve problem on unfitted grids.
In fact, we only need to update the interface elements by connecting the intersection points between background grids and fracture, thereby the resulting grids are again fitted with fracture and thus can be naturally embedded into our current framework.
Therefore, this paper focuses on the heart of the novelty on the fitted mesh to make the presentation clear.

The rest of this paper is organized as follows.
In the next section, we describe the model problem and formulate the staggered DG formulation for the bulk region coupled with standard conforming Galerkin formulation inside the fracture.
In addition, some fundamental ingredients are given in order to prove the \textit{a priori} error estimates.
In Section~\ref{sec:error}, \textit{a priori} error analysis is derived for bulk flux, bulk pressure and fracture pressure measured in $L^2$ error, where a discrete trace inequality is proved.
Then several numerical experiments are given in Section~\ref{sec:numerical} to confirm the theoretical findings, where various tests including elements with small edges and anisotropic meshes are demonstrated.
Finally, a conclusion is given.

\section{Description of staggered DG method}

In this section we first describe the governing equations modeling Darcy flows in fractured porous media.
Then staggered DG discretization is derived for the model problem under consideration.
Finally, we introduce some technical results that are vital for subsequent sections.


\subsection{Model problem}
We consider a porous medium saturated by an incompressible fluid that occupies the space region $\Omega\subset \mathbb{R}^2$ and is crossed by a single fracture $\Gamma$.
Here, $\Omega_B:=\Omega\backslash \bar{\Gamma}$ represents the bulk region and can be decomposed as $\Omega_B:=\Omega_{B,1}\cup \Omega_{B,2}$.
In addition, we denote by $\partial \Omega_B:=\bigcup_{i=1}^2 \partial \Omega_{B,i}\backslash \bar{\Gamma}$ and denote by $\partial \Gamma$ the boundary of fracture $\Gamma$.
$\bm{n}_\Gamma$ denotes a unit normal vector to $\Gamma$ with a fixed orientation.
The schematic of the bulk and fracture domain is illustrated in Figure~\ref{fig:bulkdomain}.
Without loss of generality, we assume in the following that the subdomains are numbered so that $\bm{n}_\Gamma$ coincides with the outward normal direction of $\Omega_{B,1}$.

In the bulk region, we model the motion of the incompressible fluid by Darcy's law in mixed form, so that the pressure $p: \Omega_B\rightarrow \mathbb{R}$ and the flux $\bm{u}: \Omega_B \rightarrow \mathbb{R}^2$ satisfy
\begin{align}
    \bm{u}+K\nabla p    & =\bm{0}\quad \mbox{in}\;\Omega_B,\label{eq:bulk1} \\
    \nabla \cdot \bm{u} & =f\quad \mbox{in}\;\Omega_B,\label{eq:bulk2} \\
    p                   & =p_0\quad \mbox{on}\; \partial \Omega_B.
\end{align}
Here, $p_0\in H^{\frac{1}{2}}(\partial \Omega_B)$ the boundary pressure, and $K:\Omega_B \rightarrow \mathbb{R}^{2\times 2}$ the bulk permeability tensor, which is assumed to be a symmetric, piecewise constant.
Further, we assume that $K$ is uniformly elliptic so that there exist two strictly positive real numbers $K_1$ and $K_2$ satisfying for almost every $x\in \Omega_B$ and all $z\in \mathbb{R}^2$ such that $|\bm{z}|=1$
\begin{equation*}
    0<K_1\leq K(x) \bm{z}\cdot \bm{z}\leq K_2.
\end{equation*}

Inside the fracture, we consider the motion of the fluid as governed by Darcy's law in primal form, so that the fracture pressure $p_\Gamma: \Gamma\rightarrow \mathbb{R}$ satisfies
\begin{equation}
    \begin{aligned}
        -\nabla_t \cdot (K_\Gamma \nabla_t p_\Gamma)
         & =\ell_{\Gamma}f_\Gamma+[\bm{u}\cdot \bm{n}_\Gamma]
         &                                                    & \mbox{in}\; \Gamma,          \\
        p_\Gamma
         & = g_\Gamma
         &                                                    & \mbox{on}\; \partial \Gamma,
    \end{aligned}
    \label{eq:fracture}
\end{equation}
where $f_\Gamma\in L^2(\Gamma)$ and $K_\Gamma: = \kappa_\Gamma^*\ell_\Gamma$ with $\kappa_\Gamma^*: \Gamma\rightarrow \mathbb{R}$ and $\ell_\Gamma:\Gamma\rightarrow \mathbb{R}$ denoting the tangential permeability and thickness of the fracture, respectively.
The quantities $\kappa_\Gamma^*$ and $\ell_\Gamma$ are assumed to be piecewise constants.
Here, $\nabla_t\cdot$ and $\nabla_t$ denote the tangential divergence and gradient operators along $\Gamma$, respectively.
\Red{For the sake of simplicity, we assume $p_0=0$, $g_\Gamma=0$ in the analysis.}

The above problems are coupled by the following interface conditions
\begin{equation}
    \begin{aligned}
        \eta_\Gamma \{\bm{u}\cdot\bm{n}_\Gamma\} & =[p]            &  & \mbox{on}\;\Gamma, \\
        \alpha_\Gamma[\bm{u}\cdot\bm{n}_\Gamma]  & =\{p\}-p_\Gamma &  & \mbox{on}\;\Gamma,
    \end{aligned}\label{eq:interface}
\end{equation}
where we set
\begin{equation*}
    \eta_\Gamma: =\frac{\ell_\Gamma}{\kappa_\Gamma^n},\quad \alpha_\Gamma:=\eta_\Gamma(\frac{\xi}{2}-\frac{1}{4}).
\end{equation*}
Here $\xi\in (\frac{1}{2},1]$ is a model parameter, and $\kappa_\Gamma^n: \Gamma\rightarrow \mathbb{R}$ represents the normal permeability of the fracture, which is assumed to be a piecewise constant.
As in the bulk domain, we assume that there exists positive constants $\kappa_1^*,\kappa_2^*,\kappa_1^n,\kappa_2^n$ such that, almost everywhere on $\Gamma$,
\begin{equation*}
    \kappa_1^*\leq \kappa_\Gamma^*\leq \kappa_2^*,\quad \kappa_1^n\leq \kappa_\Gamma^n\leq \kappa_2^n.
\end{equation*}
Also, $[\cdot]$ and $\{\cdot\}$ are jump and average operators, respectively, and their precise definitions can be found in the next subsection.
The well-posedness of the coupled problem for $\xi\in (\frac{1}{2},1]$ has been proved in \cite{Martin05}.
\begin{figure}
    \centering
    \includegraphics[width=0.5\textwidth]{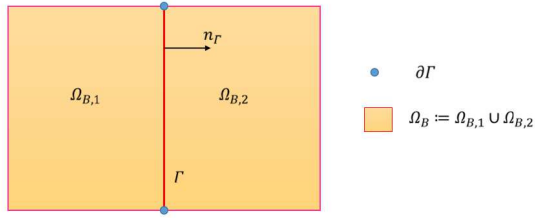}
    \caption{Illustration of bulk and fracture domain.}
    \label{fig:bulkdomain}
\end{figure}

\begin{remark}[Neumann boundary conditions]
    \rm{When the fracture tip is immersed in the domain $\Omega_B$, the boundary condition at the immersed tip can be modeled as a homogeneous Neumann boundary condition, see \cite{Angot09}.
    For both bulk and fracture domains, the Neumann boundary condition can be treated as a natural boundary condition.
    Since the analysis for such boundary condition is parallel to the analysis for Dirichlet boundary conditions, we only consider the latter one for simplicity.}
\end{remark}

Before closing this subsection, we introduce some notations that will be employed throughout the paper.
Let $D\subset \mathbb{R}^d,$ $d=1,2$, we adopt the standard notations for the Sobolev spaces $H^s(D)$ and their associated norms $\|\cdot\|_{s,D}$, and semi-norms $|\cdot|_{s,D}$ for $s\geq 0$.
The space $H^0(D)$ coincides with $L^2(D)$, for which the norm is denoted as $\|\cdot\|_{D}$.
We use $(\cdot,\cdot)_D$ to denote the inner product for $d=2$ and $\langle\cdot,\cdot\rangle_D$ for $d=1$.
If $D=\Omega$, the subscript $\Omega$ will be dropped unless otherwise mentioned.
In the sequel, we use $C$ to denote a generic positive constant which may have different values at different occurrences.

\subsection{Staggered DG method}

In this subsection, we begin with introducing the construction of our staggered DG spaces, in line with this we then present the staggered DG method for the model problem \eqref{eq:bulk1}-\eqref{eq:interface}.
We consider a family of meshes $\mathcal{T}_u$ made of disjoint polygonal (primal) elements which are aligned with the fracture $\Gamma$ so that any element $T\in \mathcal{T}_u$ can not be cut by $\Gamma$.
Note that, since $\Omega_{B,1}$ and $\Omega_{B,2}$ are disjoint, each element $T$ belongs to one of the two subdomains.
The union of all the edges excluding the edges lying on the fracture $\Gamma$ in the decomposition $\mathcal{T}_u$ is called primal edges, which is denoted as $\mathcal{F}_u$.
Here we use $\mathcal{F}_u^0$ to stand for the subset of $\mathcal{F}_u$, that is the set of edges in $\mathcal{F}_{u}$ that do not lie on $\partial\Omega_B$.
In addition, we use $\mathcal{F}_h^\Gamma$ to denote the one-dimensional mesh of the fracture $\Gamma$.
For the construction of staggered DG method, we decompose each element $T \in \mathcal{T}_u$ into the union of triangles by connecting the interior point $\nu$ of $T$ to all the vertices.
Here the interior point $\nu$ is chosen as the center point for simplicity.
We rename the union of these sub-triangles by $S(\nu)$ to indicate that the triangles sharing common vertex $\nu$.
In addition, the resulting simplicial sub-meshes are denoted as $\mathcal{T}_h$.
Moreover, some additional edges are generated in the subdivision process due to the connection of $\nu$ to all the vertices of the primal element, and these edges are denoted by  $\mathcal{F}_p$.
For each triangle $\tau\in \mathcal{T}_h$, we let $h_\tau$ be the diameter of $\tau$ and $h=\max\{h_\tau, \tau\in \mathcal{T}_h\}$.
In addition, we define $\mathcal{F}:=\mathcal{F}_{u}\cup \mathcal{F}_{p}$ and $\mathcal{F}^{0}:=\mathcal{F}_{u}^{0}\cup \mathcal{F}_{p}$.
The construction for general meshes is illustrated in Figure~\ref{grid}, where the black solid lines are edges in $\mathcal{F}_{u}$ and black dotted lines are edges in $\mathcal{F}_{p}$.

Finally, we construct the dual mesh.
For each interior edge $e\in \mathcal{F}_{u}^0$, we use $D(e)$ to represent the dual mesh, which is the union of the two triangles in $\mathcal{T}_h$ sharing the edge $e$.
For each edge $e\in(\mathcal{F}_{u}\backslash\mathcal{F}_{u}^0)\cup \mathcal{F}_h^\Gamma$, we use $D(e)$ to denote the triangle in $\mathcal{T}_h$ having the edge $e$, see Figure~\ref{grid}.

For each edge $e$, we define a unit normal vector $\bm{n}_{e}$ as follows:
If $e\in \mathcal{F}\backslash \mathcal{F}^{0}$, then $\bm{n}_{e}$ is the unit normal vector of $e$ pointing towards the outside of $\Omega$.
If $e\in \mathcal{F}^{0}$, an interior edge, we then fix $\bm{n}_{e}$ as one of the two possible unit normal vectors on $e$.
When there is no ambiguity, we use $\bm{n}$ instead of $\bm{n}_{e}$ to simplify the notation.

Typical analysis for polygonal element usually requires the following mesh regularity assumptions (cf. \cite{Beir13,Cangiani16}):
\begin{description}
    \item[Assumption (A)] Every element $S(\nu)$ in $\mathcal{T}_{u}$ is star-shaped with respect to a ball of radius $\geq \rho_S h_{S(\nu)}$, where $\rho_S$ is a positive constant and $h_{S(\nu)}$ denotes the diameter of $S(\nu)$.
    \item[Assumption (B)] For every element $S(\nu)\in \mathcal{T}_{u}$ and every edge $e\in \partial S(\nu)$, it satisfies $h_e\geq \rho_E h_{S(\nu)}$, where $\rho_E$ is a positive constant and $h_e$ denotes the length of edge $e$.
\end{description}
Assumption (A) and (B) can guarantee that the triangulation $\mathcal{T}_h$ is shape regular.
However, it excludes the elements with arbitrarily small edges, which is interesting from the practical applications.
Thus, in this paper, we will show the convergence estimates by only assuming Assumption (A).

\begin{figure}
    \centering
    \includegraphics[width=0.35\textwidth]{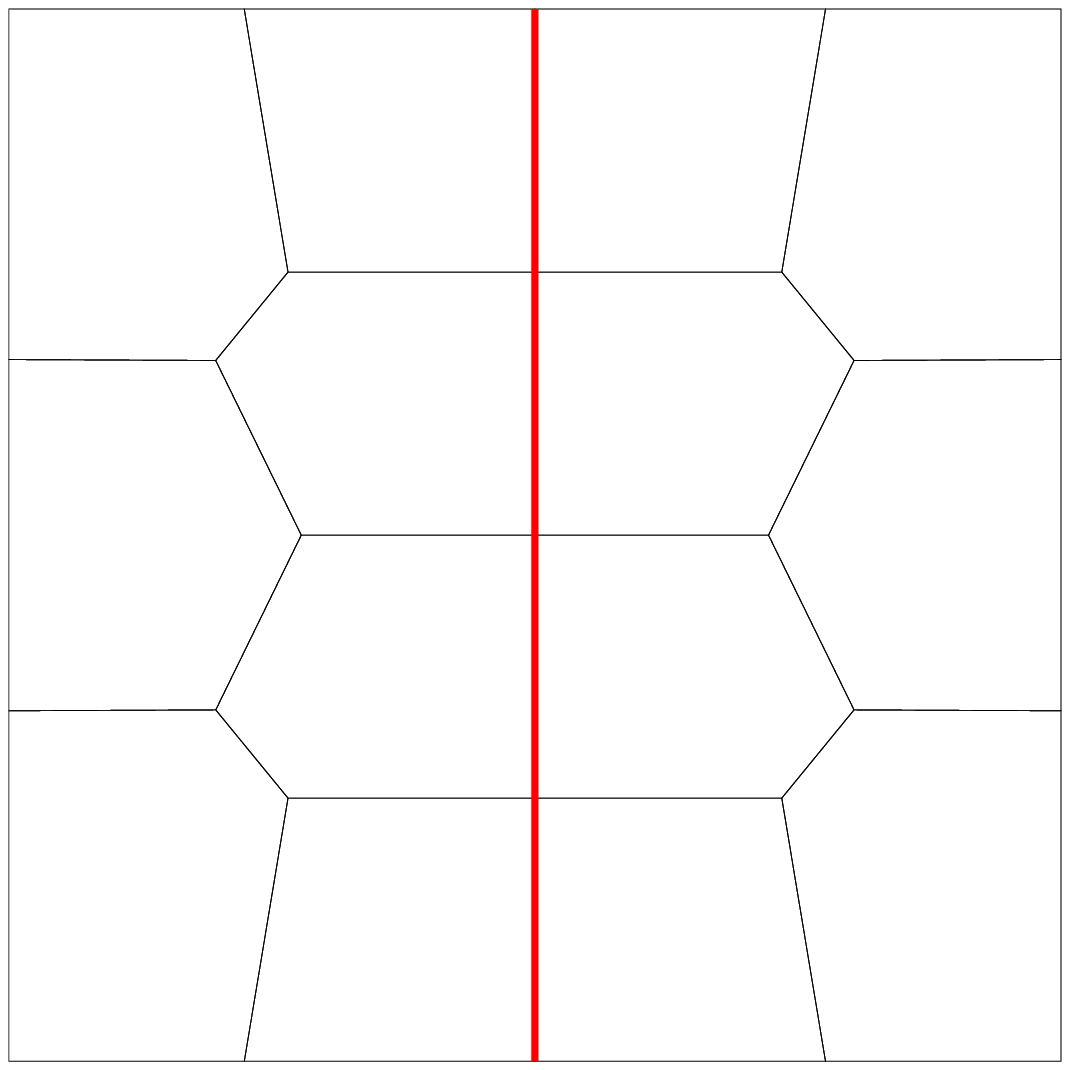}\hspace{1em}
    \includegraphics[width=0.35\textwidth]{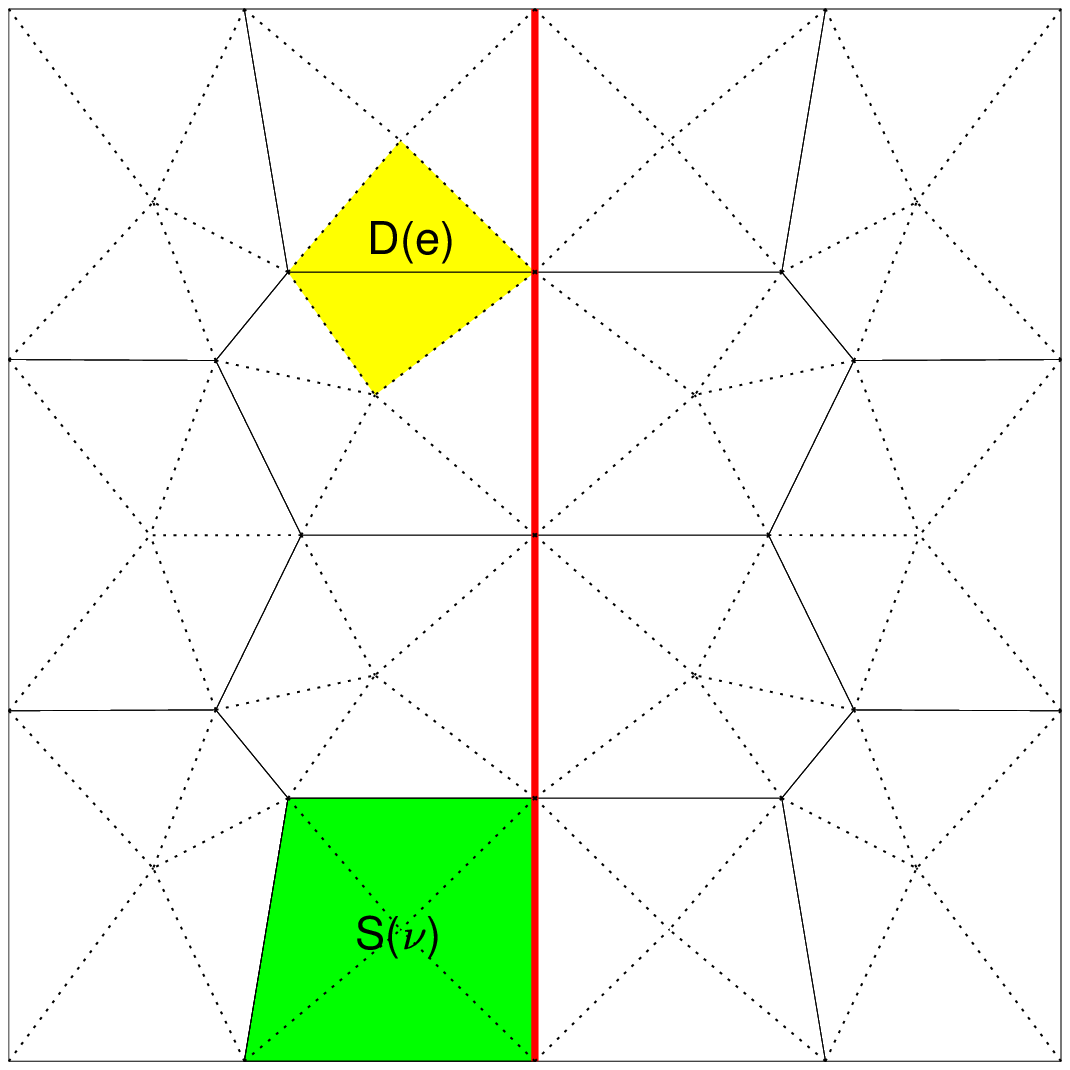}
    \caption{Schematic of the primal mesh $S(\nu)$, the dual mesh $D(e)$ and the primal simplicial sub-meshes.}
    \label{grid}
\end{figure}


Let $k\geq 0$ be the order of approximation. For every $\tau \in \mathcal{T}_{h}$ and $e\in\mathcal{F}$, we define $P^{k}(\tau)$ and $P^{k}(e)$ as the spaces of polynomials of degree less than or equal to $k$ on $\tau$ and $e$, respectively.
For $q$ and $\bm{v}$ belonging to the broken Sobolev space the jump $[q]\mid_e$ and the jump $[\bm{v}\cdot\bm{n}]\mid_e$ over $e\in \mathcal{F}^0\cup \mathcal{F}_h^\Gamma$ are defined respectively as
\begin{equation*}
    [q]=q_{1}-q_{2}, \quad [\bm{v}\cdot\bm{n}]=\bm{v}_{1}\cdot\bm{n}-\bm{v}_{2}\cdot\bm{n},
\end{equation*}
where $q_{i}=q\mid_{\tau_{i}}$, $\bm{v}_{i}=\bm{v}\mid_{\tau_{i}}$ and $\tau_{1}$, $\tau_{2}$ are the two triangles in $\mathcal{T}_h$ having the edge $e$.
Moreover, for $e\in \mathcal{F}\backslash \mathcal{F}^0$, we define $[q]=q_1$.
In the above definitions, we assume $\bm{n}$ is pointing from $\tau_1$ to $\tau_2$.

Similarly, we define the average $\{q\}\mid_e$ and the average $\{\bm{v}\cdot\bm{n}\}\mid_e$ over $e\in \mathcal{F}^0\cup \mathcal{F}_h^\Gamma$ by
\begin{align*}
    \{q\}=\frac{q_{1}+q_{2}}{2}, \quad \{\bm{v}\cdot\bm{n}\}=\frac{\bm{v}_{1}\cdot\bm{n}+\bm{v}_{2}\cdot\bm{n}}{2},
\end{align*}
where $q_{i}=q\mid_{\tau_{i}}$, $\bm{v}_{i}=\bm{v}\mid_{\tau_{i}}$ and $\tau_{1}$, $\tau_{2}$ are the two triangles in $\mathcal{T}_h$ having the edge $e$.

Next, we will introduce some finite dimensional spaces.
First, we define the following locally $H^{1}(\Omega)$ conforming space $S_h$:
\begin{equation*}
    S_{h}:=\{q :
    q\mid_{\tau}\in P^{k}(\tau)\;
    \forall \tau \in \mathcal{T}_{h};\;[q]\mid_e=0\;\forall e\in\mathcal{F}_{u}^{0};
    \;q\mid_{\partial \Omega_B}=0\}.
\end{equation*}
Notice that, if $q\in S_h$, then $q\mid_{D(e)}\in H^1(D(e))$ for each edge $e\in (\mathcal{F}_{u}\cup \mathcal{F}_h^\Gamma)$ and no continuity is imposed across $e\in \mathcal{F}_h^\Gamma$ for function $q\in S_h$.
The discrete $H^1$-norm for $S_h$ are defined as follows
\begin{equation*}
    \|q\|_Z^2=\sum_{\tau\in \mathcal{T}_h}\|\nabla q\|_{0,\tau}^2+\sum_{\tau\in \mathcal{T}_h}\sum_{e\in \mathcal{F}_{p}\cap \partial \tau}\frac{h_e}{2|\tau|}\|[q]\|_{0,e}^2.
\end{equation*}
where $|\tau|$ represents the area of triangle $\tau\in \mathcal{T}_h$.
Note that the scaling in the second term used here is different from that of \cite{LinaPark}, and this modification enables us to show the convergence estimates without Assumption (B).
We specify the degrees of freedom for $S_h$ similar to that of \cite{ChungWave}.

(SD1) For $e\in \mathcal{F}_u$, we have
\begin{equation*}
    \phi_e(q): = \langle q,p_k\rangle_e\quad \forall p_k\in P^{k}(e).
\end{equation*}

(SD2) For $\tau\in \mathcal{T}_h$, we have
\begin{equation*}
    \phi_\tau(q): = (q,p_{k-1})_\tau \quad \forall p_{k-1}\in P^{k-1}(\tau).
\end{equation*}

(SD3) For $e\in \mathcal{F}_h^\Gamma$, we have for each $i=1,2$
\begin{equation*}
    \phi_{e}^i(q): = \langle q\mid_{\Omega_{B,i}},p_k\rangle_e \quad  \forall p_k\in P^k(e).
\end{equation*}
Note that in original staggered DG method, the finite dimensional space for pressure is continuous over all the primal edges, in which case (SD3) can be compliant with (SD1).
In this paper we consider Darcy flows with fracture where the pressure is discontinuous across the fracture, thereby (SD3) can not be compliant with (SD1).
Proceeding analogously to Lemma~2.2 of \cite{ChungWave}, we can show that any function $q\in S_h$ is uniquely determined by the degrees of freedom (SD1)-(SD3), which is omitted here for simplicity.

We next define the following locally $H(\mbox{div};\Omega)-$conforming space $\bm{V}_h$:
\begin{equation*}
    \bm{V}_{h}=\{\bm{v}:
    \bm{v}\mid_{\tau} \in P^{k}(\tau)^{2}\;\forall \tau \in \mathcal{T}_{h};\;
    [\bm{v}\cdot\bm{n}]\mid_e=0\;\forall e\in \mathcal{F}_{p}\}.
\end{equation*}
Note that if $\bm{v}\in \bm{V}_h$, then $\bm{v}\mid_{S(\nu)}\in H(\textnormal{div};S(\nu))$ for each $S(\nu)\in \mathcal{T}_{u}$.
We equip $\bm{V}_h$ with the following discrete $L^2$ norm
\begin{equation*}
    \|\bm{v}\|_{X'}^2=\|\bm{v}\|_0^2+\sum_{\tau\in \mathcal{T}_h}\sum_{e\in \mathcal{F}_{p}\cap \partial \tau}\frac{|\tau|}{2h_e}\|\bm{v}\cdot \bm{n}\|_{0,e}^2.
\end{equation*}
The degrees of freedom for $\bm{V}_h$ can be defined below.

(VD1) For each edge $e\in \mathcal{F}_{p}$, we have
\begin{equation*}
    \psi_e(\bm{v}):=\langle\bm{v}\cdot\bm{n}, p_k\rangle_e \quad \forall p_k\in P^k(e).
\end{equation*}

(VD2) For each $\tau\in \mathcal{T}_h$, we have
\begin{equation*}
    \psi_\tau(\bm{v}):=(\bm{v}, \bm{p}_{k-1})_\tau \quad \forall \bm{p}_{k-1}\in P^{k-1}(\tau)^2.
\end{equation*}

Finally, we define a finite dimensional subspace of $H^1_0(\Gamma)$ by
\begin{equation*}
    W_h=\{q_\Gamma: q_\Gamma\in H^1_0(\Gamma)\;|\; q_\Gamma\mid_e \in P^k(e), \forall e\in \mathcal{F}_h^\Gamma\}.
\end{equation*}

With the above preparations, we can now derive our staggered DG method by following \cite{LinaPark,LinaParkShin}.
Multiplying \eqref{eq:bulk1} by $\bm{v}\in \bm{V}_h$ and performing integration by parts, we can obtain
\begin{align*}
     & (K^{-1}\bm{u},\bm{v})_{\Omega_B}+\sum_{e\in \mathcal{F}_u}\langle p, [\bm{v}\cdot \bm{n}]\rangle_e+\sum_{e\in \mathcal{F}_h^\Gamma}\langle [p],\{\bm{v}\cdot \bm{n}\}\rangle_e \\
     & \;+\sum_{e\in \mathcal{F}_h^\Gamma}\langle \{p\},[\bm{v}\cdot \bm{n}]\rangle_e-\sum_{\tau\in \mathcal{T}_h}(p, \nabla \cdot \bm{v})_\tau=0,
\end{align*}
where the staggered continuity property of $\bm{v}$ is integrated into the derivation.

Similarly, multiplying \eqref{eq:bulk2} by $q\in S_h$ and performing integration by parts yield
\begin{equation}
    \sum_{e\in \mathcal{F}_p}\langle\bm{u}\cdot \bm{n}, [q]\rangle_e+\sum_{e\in \mathcal{F}_h^\Gamma} \langle [\bm{u}\cdot \bm{n} ],\{q\}\rangle_e+\sum_{e\in \mathcal{F}_h^\Gamma} \langle \{\bm{u}\cdot \bm{n}\},[q]\rangle_e-\sum_{\tau\in \mathcal{T}_h}(\bm{u}, \nabla q)_\tau=(f,q)_{\Omega_B}\label{eq:uterm}.
\end{equation}
Then we exploit the interface condition \eqref{eq:interface} in \eqref{eq:uterm} and recast the above formulation as
\begin{equation*}
    \sum_{e\in \mathcal{F}_p}\langle\bm{u}\cdot \bm{n}, [q]\rangle_e+\sum_{e\in \mathcal{F}_h^\Gamma}\langle \frac{1}{\alpha_\Gamma}(\{p\}-p_\Gamma),\{q\}\rangle_e+\sum_{e\in \mathcal{F}_h^\Gamma}\langle \frac{1}{\eta_\Gamma}[p],[q]\rangle_e-\sum_{\tau\in \mathcal{T}_h}(\bm{u}, \nabla q)_\tau=(f,q)_{\Omega_B}.
\end{equation*}
As for the fracture model \eqref{eq:fracture}, we multiply by $q_{\Gamma}\in W_h$ and replace the jump term $[\bm{u}]|_\Gamma\cdot \bm{n}_\Gamma$ by utilizing \eqref{eq:interface}, which implies
\begin{equation*}
    \langle K_\Gamma\nabla_t p_{\Gamma}, \nabla_t q_\Gamma\rangle_\Gamma-\sum_{e\in \mathcal{F}_h^\Gamma}\langle\frac{1}{\alpha_\Gamma}(\{p\}-p_{\Gamma}),q_\Gamma\rangle_e
    =\langle\ell_\Gamma f_\Gamma, q_\Gamma\rangle_\Gamma.
\end{equation*}

Thereby we obtain the following discrete formulation for the model problem \eqref{eq:bulk1}-\eqref{eq:fracture}:
Find $(\bm{u}_h, p_h, p_{\Gamma,h})\in \bm{V}_h\times S_h\times W_h$ such that
\begin{equation}
    \begin{split}
        (K^{-1} \bm{u}_h, \bm{v})_{\Omega_B}+b_h^*(p_h, \bm{v})&=0,\\
        -b_h(\bm{u}_h, q)+\sum_{e\in \mathcal{F}_h^\Gamma}\langle \frac{1}{\alpha_\Gamma}(\{p_h\}-p_{\Gamma,h}),\{q\}\rangle_e+\sum_{e\in \mathcal{F}_h^\Gamma}\langle \frac{1}{\eta_\Gamma}[p_h],[q]\rangle_e&=(f,q)_{\Omega_B},\\
        \langle K_\Gamma\nabla_t p_{\Gamma,h}, \nabla_t q_{\Gamma}\rangle_\Gamma-\sum_{e\in \mathcal{F}_h^\Gamma}\langle\frac{1}{\alpha_\Gamma}(\{p_h\}-p_{\Gamma,h}),q_{\Gamma}\rangle_e
        &=\langle\ell_\Gamma f_\Gamma, q_{\Gamma}\rangle_\Gamma,\\
        \forall (\bm{v},q,q_{\Gamma})\in \bm{V}_h\times S_h\times W_h,
    \end{split}
    \label{eq:discrete}
\end{equation}
where the bilinear forms are defined by
\begin{align*}
    b_h(\bm{u}_h, q)   & =- \sum_{e\in \mathcal{F}_p}\langle \bm{u}_h\cdot\bm{n},[q]\rangle_e+\sum_{\tau\in \mathcal{T}_h}(\bm{u}_h,\nabla q)_\tau,                                                                                 \\
    b_h^*(p_h, \bm{v}) & =\sum_{e\in \mathcal{F}_u^0}\langle p_h,[\bm{v}\cdot\bm{n}]\rangle_e-\sum_{\tau\in \mathcal{T}_h}(p_h,\nabla \cdot\bm{v})_\tau+\sum_{e\in \mathcal{F}_h^\Gamma}\langle[p_h],\{\bm{v}\cdot\bm{n}\}\rangle_e
    +\sum_{e\in \mathcal{F}_h^\Gamma}\langle\{p_h\},[\bm{v}\cdot\bm{n}]\rangle_e                                                                                                                                                    \\
                       & =\sum_{e\in \mathcal{F}_u^0}\langle p_h,[\bm{v}\cdot\bm{n}]\rangle_e
    -\sum_{\tau\in \mathcal{T}_h}(p_h,\nabla \cdot\bm{v})_\tau
    +\sum_{e\in \mathcal{F}_h^\Gamma}\langle[p_h(\bm{v}\cdot\bm{n})],1\rangle_e.
\end{align*}
Summing up the equations in \eqref{eq:discrete} yields the following formulation: Find $(\bm{u}_h,p_h,p_{\Gamma,h})\in \bm{V}_h\times S_h\times W_h$ such that
\begin{equation}
    \begin{split}
        &(K^{-1}\bm{u}_h,\bm{v})_{\Omega_B}+b_h^*(p_h,\bm{v})-b_h(\bm{u}_h,q)+\sum_{e\in \mathcal{F}_h^\Gamma}\langle \frac{1}{\alpha_\Gamma}(\{p_h\}-p_{\Gamma}),\{q\}-q_{\Gamma}\rangle_e\\
        &\;+\sum_{e\in \mathcal{F}_h^\Gamma}\langle \frac{1}{\eta_\Gamma}[p_h],[q]\rangle_e
        +\langle K_\Gamma\nabla_t p_{\Gamma,h}, \nabla_t q_{\Gamma}\rangle_\Gamma=(f,q)_{\Omega_B}+\langle\ell_\Gamma f_\Gamma, q_{\Gamma}\rangle_\Gamma,\\
        &\hskip 8cm\forall (\bm{v},q,q_{\Gamma})\in \bm{V}_h\times S_h\times W_h.
    \end{split}
    \label{eq:weak}
\end{equation}
Integration by parts reveals the following adjoint property
\begin{equation}
    b_h(\bm{v},q)=b_h^*(q, \bm{v})\quad \forall (\bm{v},q)\in \bm{V}_h\times S_h.\label{eq:adjoint}
\end{equation}

\begin{remark}
    \rm{
        In the derivation we employ the interface conditions \eqref{eq:interface} to replace all the terms corresponding to $\bm{u}$ on the fracture $\Gamma$ by $p$ and $p_\Gamma$, which is different from existing methods such as the hybrid high-order method and mixed finite element method \cite{Chave18,DAngelo12}.
         \Red{Our methodology is based on the fact that the degrees of freedom for bulk pressure (SD3) are defined with respect to the primal edges on the fracture.}
        We also emphasize that the velocity $\bm{u}$ can be made both locally and globally mass conservative by a suitable postprocessing (cf. \cite{ChungCockburn14}).
        Moreover, the use of conforming finite element to discretize the equations in the fracture is made just for simplicity, other discretization techniques can be exploited.
    }
\end{remark}

\begin{lemma}Under Assumption (A), we have the following inf-sup condition
    \begin{align}
        \inf_{q\in S_h}\sup_{\bm{v}\in \bm{V}_h}\frac{b_h(\bm{v},q)}{\|\bm{v}\|_{X'}\|q\|_Z}\geq C.\label{eq:inf-sup}
    \end{align}
\end{lemma}

\begin{proof}
    The proof for this lemma follows similar idea as Theorem~3.2 of \cite{ChungWave}, we simplify the proof by direct applications of the degrees of freedom (VD1)-(VD2).
    In addition, our proof here only relies on Assumption (A) thanks to the modified norm defined for $\|\cdot\|_Z$ and $\|\cdot \|_{X'}$.

    Let $q\in S_h$. It suffices to find $\bm{v}\in \bm{V}_h$ such that
    \begin{equation*}
        b_h(\bm{v},q)\geq C \|q\|_Z^2\quad \mbox{and} \quad \|\bm{v}\|_{X'}\leq C \|q\|_Z.
    \end{equation*}
    Recall that
    \begin{equation}
        b_h(\bm{v}, q) =- \sum_{e\in \mathcal{F}_p}\langle \bm{v}\cdot\bm{n},[q]\rangle_e+\sum_{\tau\in \mathcal{T}_h}(\bm{v},\nabla q)_\tau.\label{bh-recall}
    \end{equation}
    We define $\bm{v}$ by using degrees of freedom (VD1)
    \begin{equation*}
        \langle \bm{v}\cdot \bm{n},p_k\rangle_e=\sum_{\substack{\tau\in\mathcal{T}_h\\\tau\subset D(e)}}\frac{h_e}{2|\tau|}\langle[q],p_k\rangle_e\quad\forall p_k\in P^k(e),\;e\in\mathcal{F}_p,
    \end{equation*}
    and (VD2)
    \begin{equation*}
        (\bm{v},\bm{p}_{k-1})_\tau = (\nabla q,\bm{p}_{k-1})_\tau\quad\forall\bm{p}_{k-1}\in P^{k-1}(\tau)^2,\;\tau\in\mathcal{T}_h,
    \end{equation*}
    which together with \eqref{bh-recall} yields
    \begin{equation*}
        b_h(\bm{v}, q)=- \sum_{e\in \mathcal{F}_p}\langle \bm{v}\cdot\bm{n},[q]\rangle_e+\sum_{\tau\in \mathcal{T}_h}(\bm{v},\nabla q)_\tau=\|q\|_{Z}^2.
    \end{equation*}
    On the other hand, scaling arguments imply
    \begin{equation*}
        \|\bm{v}\|_{X'}\leq C \|q\|_Z.
    \end{equation*}
    This completes the proof.



\end{proof}

Finally, we introduce the following interpolation operators, which play an important role in later analysis.
We define the interpolation operator $I_h: H^1(\Omega_B)\rightarrow S_h$ by
\begin{equation*}
    \begin{split}
        \langle I_h w-w,\psi\rangle_e&=0 \quad \forall \psi\in P^k(e),\;e\in \mathcal{F}_{u},\\
        \langle (I_hw-w)|_{\Omega_{B,i}},\psi\rangle_{e}&=0\quad \forall \psi\in P^k(e),\;e\in \mathcal{F}_h^\Gamma,\;i=1,2,\\
        (I_hw-w,\psi)_\tau&=0 \quad \forall \psi\in P^{k-1}(\tau),\;\tau\in \mathcal{T}_h
    \end{split}
\end{equation*}
and the interpolation operator $J_h: H^\delta(\Omega_B)^2\rightarrow \bm{V}_h,\delta>1/2$ by
\begin{equation*}
    \begin{split}
        \langle(J_h \bm{v}-\bm{v})\cdot \bm{n},\phi\rangle_e&=0 \quad \forall \phi\in P^k(e),\;e\in \mathcal{F}_{p},\\
        (J_h\bm{v}-\bm{v}, \bm{\phi})_\tau&=0 \quad \forall \bm{\phi}\in P^{k-1}(\tau)^2,\; \tau\in \mathcal{T}_h.
    \end{split}
\end{equation*}
The definition of the interpolation operators implies that
\begin{align*}
    b_h(\bm{u}-J_h\bm{u},w) & =0 \quad \forall w\in S_h,                 \\
    b_h^*(p-I_hp,\bm{v})    & =0 \quad \forall \bm{v}\in \bm{V}_h.
\end{align*}
Notice that if $w$ is continuous on $\Gamma$, then $I_hw$ is also continuous on $\Gamma$.
The interpolation operators $I_h$ and $J_h$ satisfy: (1) they are locally defined for each element $\tau\in \mathcal{T}_h$; (2) for $p_k\in P^k(\tau)$ and $\bm{p}_k\in P^{k}(\tau)^2$, we have $I_hp_k=p_k$ and $J_h\bm{p}_k = \bm{p}_k$.

The following error estimates are clearly satisfied on the reference element $\hat{\tau}$ by using (1) and (2) (see \cite{Ciarlet78}).
\begin{equation*}
    \begin{split}
        \|\hat{\bm{v}}-J_h\hat{\bm{v}}\|_{0,\hat{\tau}}&\leq C \|\hat{\bm{v}}\|_{k+1,\hat{\tau}},\\
        \|\hat{q}-I_h\hat{q}\|_{0,\hat{\tau}}&\leq C \|\hat{q}\|_{k+1,\hat{\tau}},\\
        \|\nabla (\hat{q}-I_h\hat{q})\|_{0,\hat{\tau}}&\leq C \|\hat{q}\|_{k+1,\hat{\tau}},
    \end{split}
\end{equation*}
where $\hat{\bm{v}}$ and $\hat{q}$ are the corresponding variables of $\bm{v}$ and $q$ on the reference element $\hat{\tau}$.
In addition, under Assumption (A), the maximum angles in $\mathcal{T}_h$ are uniformly bounded away from $\pi$ (although shape regularity is not guaranteed).
Then we can proceed as Theorem~2.1 of \cite{Apel99} to obtain the following anisotropic error estimates for $\tau\in \mathcal{T}_{h}$
\begin{equation}
    \begin{split}
        \|\bm{v}-J_h\bm{v}\|_{0,\tau}&\leq C h_\tau^{k+1}\|\bm{v}\|_{k+1,\tau}\quad \forall \bm{v}\in H^{k+1}(\tau)^2,\\
        \|q-I_hq\|_{0,\tau}&\leq C h_\tau^{k+1}\|q\|_{k+1,\tau}\quad \forall q\in H^{k+1}(\tau).
    \end{split}
    \label{eq:interpolationIhJh-local}
\end{equation}
Here, generic constants $C$ possibly depend on $\rho_S$ in Assumption~(A) but not on $\rho_E$ in Assumption~(B).
We next introduce the standard nodal interpolation operator $\pi_h: H^1(\Gamma)\rightarrow W_h$, which satisfies for $q_\Gamma\in H^{k+1}(\Gamma)$
\begin{equation*}
    \begin{split}
        \|q_\Gamma-\pi_h q_\Gamma\|_{0,e}&\leq C h_e^{k+1}\|q_\Gamma\|_{k+1,e},\\
        \|\nabla_t (q_\Gamma-\pi_h q_\Gamma)\|_{0,e}&\leq C h_e^{k}\|q_\Gamma\|_{k+1,e},
    \end{split}
\end{equation*}
where $e\in \mathcal{F}_h^\Gamma$.


\section{Error analysis}\label{sec:error}

In this section, we present the unique solvability of the discrete system \eqref{eq:discrete} and the convergence estimates for all the variables involved under Assumption (A).
As $L^2$ error of $\bm{u}$ is coupled with energy error of $p_\Gamma$, it will yield sub-optimal convergence if standard interpolation operator for $p_\Gamma$ is exploited.
As such, we propose to employ the Ritz projection which enables us to achieve the optimal convergence estimates.

\begin{theorem}[stability]\label{thm:stability}
    Under Assumption (A), the discrete system \eqref{eq:discrete} admits a unique solution $(\bm{u}_h, p_h, p_{\Gamma,h})\in \bm{V}_h\times S_h\times W_h$.
    Furthermore, there exists a positive constant independent of $h$ but possibly depending on $\rho_S$ and the problem data such that
    \begin{equation}
        \begin{split}
            &\|K^{-\frac{1}{2}}\bm{u}_h\|_{0,\Omega_B}^2+\|p_h\|_{0,\Omega_B}^2+\sum_{e\in \mathcal{F}_h^\Gamma}\|\eta_\Gamma^{-\frac{1}{2}}[p_h]\|_{0,e}^2
            +\|K_\Gamma^{\frac{1}{2}}\nabla_t p_{\Gamma,h}\|_{0,\Gamma}^2+\sum_{e\in \mathcal{F}_h^\Gamma}\|\alpha_\Gamma^{-\frac{1}{2}}(\{p_h\}-p_{\Gamma,h})\|_{0,e}^2\\
            &\leq C \Big(\|f\|_{0,\Omega_B}^2+\|\ell_\Gamma f_\Gamma\|_{0,\Gamma}^2\Big).
        \end{split}
        \label{eq:stability}
    \end{equation}

\end{theorem}

\begin{proof}

    Since \eqref{eq:weak} is a square linear system, existence follows from uniqueness, thus, it suffices to show uniqueness. Taking $\bm{v}=\bm{u}_h,q=p_h,q_{\Gamma}=p_{\Gamma,h}$ in \eqref{eq:weak} yields
    \begin{align*}
        & \|K^{-\frac{1}{2}}\bm{u}_h\|_{0,\Omega_B}^2+\|K_\Gamma^{\frac{1}{2}}\nabla_t p_\Gamma\|_{0,\Gamma}^2
        +\sum_{e\in \mathcal{F}_h^\Gamma}\|\alpha_\Gamma^{-\frac{1}{2}}(\{p_h\}-p_{\Gamma,h})\|_{0,e}^2
        +\sum_{e\in \mathcal{F}_h^\Gamma}\|\eta_\Gamma^{-\frac{1}{2}}[p_h]\|_{0,e}^2 \\
        & \leq C \Big(\|f\|_{0,\Omega_B}\|p_h\|_{0,\Omega_B}+\|\ell_\Gamma f_\Gamma\|_{0,\Gamma}\|p_{\Gamma,h}\|_{0,\Gamma}\Big).
    \end{align*}
    On the other hand, an application of the discrete Poincar\'{e}-Friedrichs inequality on anisotropic meshes (cf. \cite{DuanTan11}) leads to
    \begin{equation*}
        \|p_h\|_{0,\Omega_B}\leq C \|p_h\|_Z.
    \end{equation*}
    In view of the inf-sup condition \eqref{eq:inf-sup}, the discrete adjoint property \eqref{eq:adjoint},  and \eqref{eq:discrete}, we have
    \begin{equation*}
        C||p_h||_{0,\Omega_B}\leq C\|p_h\|_Z\leq \sup_{\bm{v}_h\in \bm{V}_h}\frac{b_h(\bm{v}_h, p_h)}{\|\bm{v}_h\|_{0,\Omega_B}}=\sup_{\bm{v}_h\in \bm{V}_h} \frac{b_h^*(p_h,\bm{v}_h)}{\|\bm{v}_h\|_{0,\Omega_B}}=\sup_{\bm{v}\in V_h}\frac{(K^{-1}\bm{u}_h,\bm{v}_h)}{\|\bm{u}_h\|_{0,\Omega_B}}\leq \|K^{-1}\bm{u}_h\|_{0,\Omega_B}.
    \end{equation*}
    Moreover, $p_{\Gamma,h}\in H^1_0(\Gamma)$ and the Poincar\'{e} inequality imply that
    \begin{equation*}
        \|p_{\Gamma,h}\|_{0,\Gamma}\leq C \|\nabla_t p_{\Gamma,h}\|_{0,\Gamma}.
    \end{equation*}
    Combining the above estimates with Young's inequality, we can infer that
    \begin{align*}
         & \|K^{-\frac{1}{2}}\bm{u}_h\|_{0,\Omega_B}^2+\|p_h\|_{0,\Omega_B}+\|K_\Gamma^{\frac{1}{2}}\nabla_t p_{\Gamma,h}\|_{0,\Gamma}^2
        +\sum_{e\in \mathcal{F}_h^\Gamma}\|\eta_\Gamma^{-\frac{1}{2}}[p_h]\|_{0,e}^2+\sum_{e\in \mathcal{F}_h^\Gamma}\|\alpha_\Gamma^{-\frac{1}{2}}(\{p_h\}-p_{\Gamma,h})\|_{0,e}^2 \\
         & \leq C \Big(\|f\|_{0,\Omega_B}^2+\|\ell_\Gamma f_\Gamma\|_{0,\Gamma}^2\Big),
    \end{align*}
    which gives the desired estimate \eqref{eq:stability}.
    Here, $C$ depends on the permeability $K$ and $K_\Gamma$.
    The uniqueness follows immediately by setting $f=f_\Gamma=0$.
\end{proof}

Here, we introduce the Ritz projection $\Pi_h^p p_\Gamma\in W_h$, which is defined by
\begin{equation}
    \langle K_\Gamma\nabla_t \Pi_h^pp_\Gamma, \nabla_t q_{\Gamma,h}\rangle_\Gamma
    = \langle K_\Gamma\nabla_t p_\Gamma, \nabla_t q_{\Gamma,h}\rangle_\Gamma
    \quad \forall q_{\Gamma,h}\in W_h.\label{eq:ritz-projection}
\end{equation}
It is well-posed by the Riesz representation theorem.
Then taking $q_{\Gamma,h}=\pi_hp_\Gamma-\Pi_hp_\Gamma$ in \eqref{eq:ritz-projection} yields
\begin{equation*}
    \|K_\Gamma^{1/2}\nabla_t (p_\Gamma-\Pi_h^pp_\Gamma)\|_{0,\Gamma}^2
    =\langle K_\Gamma \nabla_t ( p_\Gamma-\Pi_h^pp_\Gamma), \nabla_t (p_\Gamma-\pi_hp_\Gamma)\rangle_\Gamma,
\end{equation*}
which implies
\begin{equation*}
    \|K_\Gamma^{\frac{1}{2}}\nabla_t (p_\Gamma-\Pi_h^pp_\Gamma)\|_{0,\Gamma}
    \leq \|K_\Gamma^{\frac{1}{2}}\nabla_t (p_\Gamma-\pi_hp_\Gamma)\|_{0,\Gamma}
    \leq C \Big(\sum_{e\in \mathcal{F}_h^\Gamma}h_e^{2k}\|K_{\Gamma}^{\frac{1}{2}}p_\Gamma\|_{k+1,e}^2\Big)^{1/2}.
\end{equation*}
Next, we show the $L^2$ error estimate for $\|p_\Gamma-\Pi_h^pp_\Gamma\|_{0,\Gamma}$. Consider the dual problem
\begin{equation}
    \begin{aligned}
        -\nabla_t \cdot (K_\Gamma\nabla_t \phi )
        &= p_\Gamma-\Pi_h^p p_\Gamma
        &&\mbox{on} \;\Gamma,\\
        \phi
        &=0
        &&\mbox{on}\;\partial \Gamma,
    \end{aligned}\label{eq:dual-poissonfracture1}
\end{equation}
which satisfies the following elliptic regularity estimate (cf. \cite{ChuHou10})
\begin{align*}
    (\sum_{e\in \mathcal{F}_h^\Gamma}\|K_\Gamma\phi\|_{2,e}^2)^{1/2}\leq C \|p_\Gamma-\Pi_h^p p_\Gamma\|_{0,\Gamma}.
\end{align*}
Multiplying \eqref{eq:dual-poissonfracture1} by $p_\Gamma-\Pi_h^p p_\Gamma$ and integration by parts, we can obtain
\begin{align*}
\|p_\Gamma-\Pi_h^p p_\Gamma\|_{0,\Gamma}^2= \langle K_\Gamma \nabla_t \phi, \nabla_t (p_\Gamma-\Pi_h^p p_\Gamma)\rangle_\Gamma.
\end{align*}
Owing to \eqref{eq:ritz-projection}, we can bound the above equation by
\begin{align*}
    \|p_\Gamma-\Pi_h^p p_\Gamma\|_{0,\Gamma}^2
    &= \langle K_\Gamma\nabla_t (\phi-\pi_h\phi), \nabla_t (p_\Gamma-\Pi_h^p p_\Gamma)\rangle_\Gamma\\
    &\leq \|K_\Gamma\nabla_t (\phi-\pi_h\phi)\|_{0,\Gamma}\|\nabla_t (p_\Gamma-\Pi_h^p p_\Gamma)\|_{0,\Gamma}\\
    &\leq C h(\sum_{e\in \mathcal{F}_h^\Gamma}\|K_{\Gamma}\phi\|_{2,e}^2)^{\frac{1}{2}}\|\nabla_t (p_\Gamma-\Pi_h^p p_\Gamma)\|_{0,\Gamma}\\
    &\leq C h\|p_\Gamma-\Pi_h^p p_\Gamma\|_{0,\Gamma}\|\nabla_t (p_\Gamma-\Pi_h^p p_\Gamma)\|_{0,\Gamma}.
\end{align*}
Thus
\begin{align}
\|p_\Gamma-\Pi_h^p p_\Gamma\|_{0,\Gamma}\leq C K_{\Gamma,\min}^{-\frac{1}{2}}h^{k+1}\Big(\sum_{e\in \mathcal{F}_h^\Gamma}\|K_{\Gamma}^{\frac{1}{2}}p_\Gamma\|_{k+1,e}^2\Big)^{1/2}.\label{eq:dual-L2f}
\end{align}
With help of the Ritz projection, we derive \textit{a priori} error estimates.
Note that the following theorem states the optimal convergence for $L^2$ error of the flux, $||K^{-\frac{1}{2}}(\bm{u}-\bm{u}_h)||_{0,\Omega}$, and superconvergence for semi-$H^1$ error of the pressure on the fracture, $||K^{\frac{1}{2}}_\Gamma\nabla_t(\Pi_h^pp_\Gamma-p_{\Gamma,h})||_{0,\Gamma}$.

\begin{theorem}\label{thm:energyError}
    Under Assumption (A), there exists a positive constant $C$ independent of $h$ and of the problem data, but possibly depending on $\rho_S$ such that
    \begin{equation*}
        \begin{aligned}
         & \|K^{-\frac{1}{2}}(J_h \bm{u}-\bm{u}_h)\|_{0,\Omega_B}+\|K_\Gamma^{\frac{1}{2}}\nabla_t (\Pi_h^p  p_\Gamma-p_{\Gamma,h})\|_{0,\Gamma}\\
         &+\Big(\sum_{e\in \mathcal{F}_h^\Gamma}\|\eta_\Gamma^{-\frac{1}{2}}[I_hp-p_h]\|_{0,e}^2\Big)^{\frac{1}{2}}+\Big(\sum_{e\in \mathcal{F}_h^\Gamma}\|\alpha_\Gamma^{-\frac{1}{2}}(\{I_hp-p_h\}-(\Pi_h^p  p_\Gamma-p_{\Gamma,h}))\|_{0,e}^2\Big)^{\frac{1}{2}}\\
         & \;\leq C \Big(\|K^{-\frac{1}{2}}(\bm{u}-J_h\bm{u})\|_{0,\Omega_B}^2
        +\|\alpha_\Gamma^{-\frac{1}{2}}(p_\Gamma-\Pi_h^p p_\Gamma)\|_{0,\Gamma}^2\Big).
        \end{aligned}
    \end{equation*}
\end{theorem}

\begin{proof}
    \Red{Our discrete formulation \eqref{eq:discrete} is consistent due to its derivation.}
    Thereby we can obtain the following error equations
    \begin{align}
        (K^{-1}(\bm{u}-\bm{u}_h), \bm{v})_{\Omega_B}+b_h^*(p-p_h,\bm{v})
         & =0,\label{eq:err1}   \\
        -b_h(\bm{u}-\bm{u}_h, q)
        +\sum_{e\in \mathcal{F}_h^\Gamma}\langle \frac{1}{\alpha_\Gamma}(\{p-p_h\}-(p_\Gamma-p_{\Gamma,h})),\{q\}\rangle_e
        +\sum_{e\in \mathcal{F}_h^\Gamma}\langle \frac{1}{\eta_\Gamma}[p-p_h],[q]\rangle_e
         & =0\;,\label{eq:err2} \\
        \langle K_\Gamma \nabla_t (p_\Gamma-p_{\Gamma,h}), \nabla_t q_{\Gamma}\rangle_\Gamma
        -\sum_{e\in \mathcal{F}_h^\Gamma}\langle\frac{1}{\alpha_\Gamma}\{p-p_h\}, q_{\Gamma}\rangle_e
        +\sum_{e\in \mathcal{F}_h^\Gamma}\langle\frac{1}{\alpha_\Gamma}( p_\Gamma-p_{\Gamma,h}), q_{\Gamma}\rangle_e
         & =0,\label{eq:err3}   \\
        \forall (\bm{v}, q, q_{\Gamma})\in \bm{V}_h\times S_h\times W_h.
    \end{align}

    Taking $\bm{v}=J_h \bm{u}-\bm{u}_h, q= I_hp-p_h, q_{\Gamma} = \Pi_h^p p_\Gamma-p_{\Gamma,h}$ in \eqref{eq:err1}-\eqref{eq:err3} and adding the resulting equations, we can obtain
    \begin{equation}
        \begin{aligned}
             & (K^{-1} (J_h\bm{u}-\bm{u}_h),J_h\bm{u}-\bm{u}_h)_{\Omega_B}
            +\langle K_\Gamma \nabla_t (\Pi_h^p  p_\Gamma-p_{\Gamma,h}),\nabla_t (\Pi_h^p  p_\Gamma-p_{\Gamma,h})\rangle_\Gamma \\
             & \;+\sum_{e\in \mathcal{F}_h^\Gamma}\langle\frac{1}{\eta_\Gamma}[p-p_h],[I_hp-p_h]\rangle_e
            +\sum_{e\in \mathcal{F}_h^\Gamma}\langle\frac{1}{\alpha_\Gamma}(\{p-p_h\}-(p_\Gamma-p_{\Gamma,h})),
            \{I_hp-p_h\}-(\Pi_h^p p_\Gamma-p_{\Gamma,h})\rangle_e                                                       \\
             & \qquad=(K^{-1}(J_h\bm{u}-\bm{u}),J_h\bm{u}-\bm{u}_h)_{\Omega_B}
            +\langle K_\Gamma \nabla_t (\Pi_h^p  p_\Gamma-p_{\Gamma}),\nabla_t (\Pi_h^p  p_\Gamma-p_{\Gamma,h})\rangle_\Gamma.
        \end{aligned}
        \label{eq:error-divide}
    \end{equation}
    It then follows from the Cauchy-Schwarz inequality and the properties of $I_h$ nd $J_h$ that
    \begin{align*}
         & \|K^{-\frac{1}{2}}(J_h \bm{u}-\bm{u}_h)\|_{0,\Omega_B}^2
        +\|K_\Gamma^{\frac{1}{2}}\nabla_t (\Pi_h^p  p_\Gamma-p_{\Gamma,h})\|_{0,\Gamma}^2
        +\sum_{e\in \mathcal{F}_h^\Gamma}\|\eta_\Gamma^{-\frac{1}{2}}[I_hp-p_h]\|_{0,e}^2                                             \\
         & \;+\sum_{e\in \mathcal{F}_h^\Gamma}\|\alpha_\Gamma^{-\frac{1}{2}}(\{I_hp-p_h\}-(\Pi_h^p  p_\Gamma-p_{\Gamma,h}))\|_{0,e}^2 \\
         & =(K^{-1}(J_h\bm{u}-\bm{u}),J_h\bm{u}-\bm{u}_h)_{\Omega_B}
        +\langle K_\Gamma\nabla_t (\Pi_h^p  p_\Gamma- p_{\Gamma}),\nabla_t (\Pi_h^p  p_\Gamma-p_{\Gamma,h})\rangle_\Gamma                \\
         & \;+\sum_{e\in \mathcal{F}_h^\Gamma}\langle \frac{1}{\alpha_\Gamma}(p_\Gamma-\Pi_h^p  p_\Gamma),
        \{I_hp-p_h\}-(\Pi_h^p  p_\Gamma-p_{\Gamma,h})\rangle_e                                                                          \\
         & \leq C \Big(\|K^{-\frac{1}{2}}(J_h\bm{u}-\bm{u})\|_{0,\Omega_B} \|K^{-\frac{1}{2}}(J_h\bm{u}-\bm{u}_h)\|_{0,\Omega_B}   \\
         & \;+\sum_{e\in \mathcal{F}_h^\Gamma}\|\alpha_\Gamma^{-\frac{1}{2}}(\{I_hp-p_h\}-(\Pi_h^p  p_\Gamma-p_{\Gamma,h}))\|_{0,e}
        \|\alpha_\Gamma^{-\frac{1}{2}}(p_\Gamma-\Pi_h^p  p_\Gamma)\|_{0,e}\Big).
    \end{align*}
    Combined with Young's inequality, this leads to
    \begin{align*}
         & \|K^{-\frac{1}{2}}(J_h \bm{u}-\bm{u}_h)\|_{0,\Omega_B}^2
        +\|K_\Gamma^{\frac{1}{2}}\nabla_t (\Pi_h^p  p_\Gamma-p_{\Gamma,h})\|_{0,\Gamma}^2
        +\sum_{e\in \mathcal{F}_h^\Gamma}\|\eta_\Gamma^{-\frac{1}{2}}[I_hp-p_h]\|_{0,e}^2                                        \\
         & \;+\sum_{e\in \mathcal{F}_h^\Gamma}\|\alpha_\Gamma^{-\frac{1}{2}}(\{I_hp-p_h\}-(\Pi_h^p  p_\Gamma-p_{\Gamma,h}))\|_{0,e}^2 \\
         & \;\leq C \Big(\|K^{-\frac{1}{2}}(\bm{u}-J_h\bm{u})\|_{0,\Omega_B}^2
        +\|\alpha_\Gamma^{-\frac{1}{2}}(p_\Gamma-\Pi_h^p p_\Gamma)\|_{0,\Gamma}^2\Big).
    \end{align*}
    Therefore, the proof is completed.
\end{proof}

\begin{corollary}\label{cor:L2err}
     \Red{Assume that $(\bm{u}\mid_\tau, p\mid_\tau,p_{\Gamma}\mid_e)\in H^{k+1}(\tau)^2\times H^{k+1}(\tau)\times H^{k+1}(e)$ for $\tau\in \mathcal{T}_h$ and $e\in \mathcal{F}_h^\Gamma$}. Then under the assumption of Theorem~\ref{thm:energyError}, there exists a positive constant $C$ independent of $h$ and of the problem data, but possibly depending on $\rho_S$ such that
    \begin{align*}
     \|K^{-\frac{1}{2}}(\bm{u}-\bm{u}_h)\|_{0,\Omega_B}&\leq C   \Big(\sum_{\tau\in \mathcal{T}_h} K_\tau^{-1}h_\tau^{2(k+1)}\|\bm{u}\|_{k+1,\tau}^2 +\alpha_1^{-1}K_{\Gamma,\min}^{-1}h^{2(k+1)}\Big(\sum_{e\in \mathcal{F}_h^\Gamma}\|K_{\Gamma}^{\frac{1}{2}}p_\Gamma\|_{k+1,e}^2\Big)\Big)^{\frac{1}{2}},\\
     \|p_\Gamma-p_{\Gamma,h}\|_{0,\Gamma}&\leq C\Big(K_{\Gamma,\min}^{-\frac{1}{2}}\Big(\sum_{\tau\in \mathcal{T}_h} K_\tau^{-1}h_\tau^{2(k+1)}\|\bm{u}\|_{k+1,\tau}^2+
        \alpha_1^{-1}\sum_{e\in\mathcal{F}_h^\Gamma}h_e^{2(k+1)}\|p_\Gamma\|_{k+1,e}^2\Big)^{\frac{1}{2}}\\
        &\;+K_{\Gamma,\min}^{-\frac{1}{2}}h^{k+1}\Big(\sum_{e\in \mathcal{F}_h^\Gamma}\|K_{\Gamma}^{\frac{1}{2}}p_\Gamma\|_{k+1,e}^2\Big)^{\frac{1}{2}}\Big)
     \end{align*}
    and
    \begin{align*}
        \|p-p_h\|_{0,\Omega_B}&\leq C \Big(K_1^{-1}\Big(\sum_{\tau\in \mathcal{T}_h} K_\tau^{-1}h_\tau^{2(k+1)}\|\bm{u}\|_{k+1,\tau}^2 +\alpha_1^{-1}\sum_{e\in\mathcal{F}_h^\Gamma}h_e^{2(k+1)}\|p_\Gamma\|_{k+1,e}^2\Big)+\sum_{\tau\in \mathcal{T}_h}h_\tau^{2(k+1)}\|p\|_{k+1,\tau}^2\Big)^{\frac{1}{2}},
    \end{align*}
    where $\alpha_1:=\min\{\alpha_\Gamma\}$, $K_{\Gamma,\min}:=\min\{K_\Gamma\}$ and $K_\tau$ is the smallest eigenvalue of $K\mid_\tau$.
    In addition, we also have the following superconvergent results
    \begin{align*}
        \|I_hp-p_h\|_Z\leq C K_1^{-\frac{1}{2}}  \Big(\sum_{\tau\in \mathcal{T}_h} K_\tau^{-1}h_\tau^{2(k+1)}\|\bm{u}\|_{k+1,\tau}^2 +\alpha_1^{-1}K_{\Gamma,\min}^{-1}h^{2(k+1)}\Big(\sum_{e\in \mathcal{F}_h^\Gamma}\|K_{\Gamma}^{\frac{1}{2}}p_\Gamma\|_{k+1,e}^2\Big)\Big)^{\frac{1}{2}}.
    \end{align*}

\end{corollary}

\begin{proof}
    Since $\Pi_h^p p_\Gamma-p_{\Gamma,h}$ belongs to $H^1_0(\Gamma)$, we have from Poincar\'{e} inequality and Theorem~\ref{thm:energyError} that
    \begin{align*}
       \|\Pi_h^pp_\Gamma-p_{\Gamma,h}\|_{0,\Gamma}\leq C\|\nabla_t(\Pi_h^pp_\Gamma-p_{\Gamma,h})\|_{0,\Gamma}\leq CK_{\Gamma,\min}^{-\frac{1}{2}}\|K_\Gamma^{\frac{1}{2}}\nabla_t(\Pi_h^pp_\Gamma-p_{\Gamma,h})\|_{0,\Gamma},
    \end{align*}
    which together with \eqref{eq:dual-L2f} and Theorem~\ref{thm:energyError} implies
    \begin{align*}
        \|p_\Gamma-p_{\Gamma,h}\|_{0,\Gamma}&\leq C\Big(K_{\Gamma,\min}^{-\frac{1}{2}}\Big(\sum_{\tau\in \mathcal{T}_h} K_\tau^{-1}h_\tau^{2(k+1)}\|\bm{u}\|_{k+1,\tau}^2+
        \alpha_1^{-1}\sum_{e\in\mathcal{F}_h^\Gamma}h_e^{2(k+1)}\|p_\Gamma\|_{k+1,e}^2\Big)^{\frac{1}{2}}\\
        &\;+K_{\Gamma,\min}^{-\frac{1}{2}}h^{k+1}\Big(\sum_{e\in \mathcal{F}_h^\Gamma}\|K_{\Gamma}^{\frac{1}{2}}p_\Gamma\|_{k+1,e}^2\Big)^{1/2}\Big).
    \end{align*}
    In addition, we also have from \eqref{eq:interpolationIhJh-local} and Theorem~\ref{thm:energyError} that
    \begin{align*}
        \|K^{-\frac{1}{2}}(\bm{u}-\bm{u}_h)\|_{0,\Omega_B}\leq C   \Big(\sum_{\tau\in \mathcal{T}_h} K_\tau^{-1}h_\tau^{2(k+1)}\|\bm{u}\|_{k+1,\tau}^2 +\alpha_1^{-1}K_{\Gamma,\min}^{-1}h^{2(k+1)}\Big(\sum_{e\in \mathcal{F}_h^\Gamma}\|K_{\Gamma}^{\frac{1}{2}}p_\Gamma\|_{k+1,e}^2\Big)\Big)^{\frac{1}{2}}.
    \end{align*}
    Finally, it follows from the inf-sup condition \eqref{eq:inf-sup} and \eqref{eq:err1} that
    \begin{align*}
        \|I_hp-p_h\|_Z&\leq C\|K^{-1}(\bm{u}-\bm{u}_h)\|_0\\
        &\leq C K_1^{-\frac{1}{2}}  \Big(\sum_{\tau\in \mathcal{T}_h} K_\tau^{-1}h_\tau^{2(k+1)}\|\bm{u}\|_{k+1,\tau}^2 +\alpha_1^{-1}K_{\Gamma,\min}^{-1}h^{2(k+1)}\Big(\sum_{e\in \mathcal{F}_h^\Gamma}\|K_{\Gamma}^{\frac{1}{2}}p_\Gamma\|_{k+1,e}^2\Big)\Big)^{\frac{1}{2}},
    \end{align*}
    which can be combined with the discrete Poincar\'{e} inequality yields
    \begin{align*}
        \|I_hp-p_h\|_{0,\Omega_B}&\leq C \|I_hp-p_h\|_Z\\
        &\leq C K_1^{-\frac{1}{2}}  \Big(
            \sum_{\tau\in \mathcal{T}_h} K_\tau^{-1}h_\tau^{2(k+1)}\|\bm{u}\|_{k+1,\tau}^2
            +\alpha_1^{-1}K_{\Gamma,\min}^{-1}h^{2(k+1)}\Big(\sum_{e\in \mathcal{F}_h^\Gamma}\|K_{\Gamma}^{\frac{1}{2}}p_\Gamma\|_{k+1,e}^2\Big)\Big)^{\frac{1}{2}}.
    \end{align*}
    Thus
    \begin{align*}
        \|p-p_h\|_{0,\Omega_B}&\leq C \Big(K_1^{-1}  \Big(\sum_{\tau\in \mathcal{T}_h} K_\tau^{-1}h_\tau^{2(k+1)}\|\bm{u}\|_{k+1,\tau}^2 +\alpha_1^{-1}K_{\Gamma,\min}^{-1}h^{2(k+1)}\Big(\sum_{e\in \mathcal{F}_h^\Gamma}\|K_{\Gamma}^{\frac{1}{2}}p_\Gamma\|_{k+1,e}^2\Big)\Big)\\
        &\;+
        \sum_{\tau\in\mathcal{T}_h}h_\tau^{2(k+1)}\|p\|_{k+1,\tau}^2\Big)^{\frac{1}{2}}.
    \end{align*}
\end{proof}

\begin{remark}
\rm{
The introduction of the Ritz projection is important to deliver optimal convergence estimates. Our methodology is based on the observation that the second term on the right side of \eqref{eq:error-divide} is actually the troublemaker, which can be cancelled if the Ritz projection is exploited. On the other hand, thanks to the Ritz projection we can achieve the optimal convergence estimates in $L^2$ errors of bulk pressure and fracture pressure without resort to duality argument as usually required by standard methods.
Alternatively, we obtain optimal convergence estimates for all the variables, which are fully robust with respect to the heterogeneity of $K$ and $K_\Gamma$, and the anisotropy of the bulk permeability. In addition, our analysis weakens the usual assumption on the polygonal mesh. All these desirable features make our method a good candidate for practical applications.
}
\end{remark}


\section{Numerical experiments}\label{sec:numerical}

In this section, we will present several numerical experiments to confirm the validity of the \textit{a priori} error estimates that we have derived for our method.
To demonstrate the robustness of our method with respect to general meshes, we employ regular polygonal grids, polygonal grids with small edges and anisotropic grids.
In addition, to further verify that our method can handle more complicated problems, we test the performances of our method on the case that the background grid is not aligned with the fracture.
%
%

\subsection{Convergence test}\label{subsec:example1}

\begin{figure}[t]
    \centering
    \includegraphics[width=0.45\textwidth]{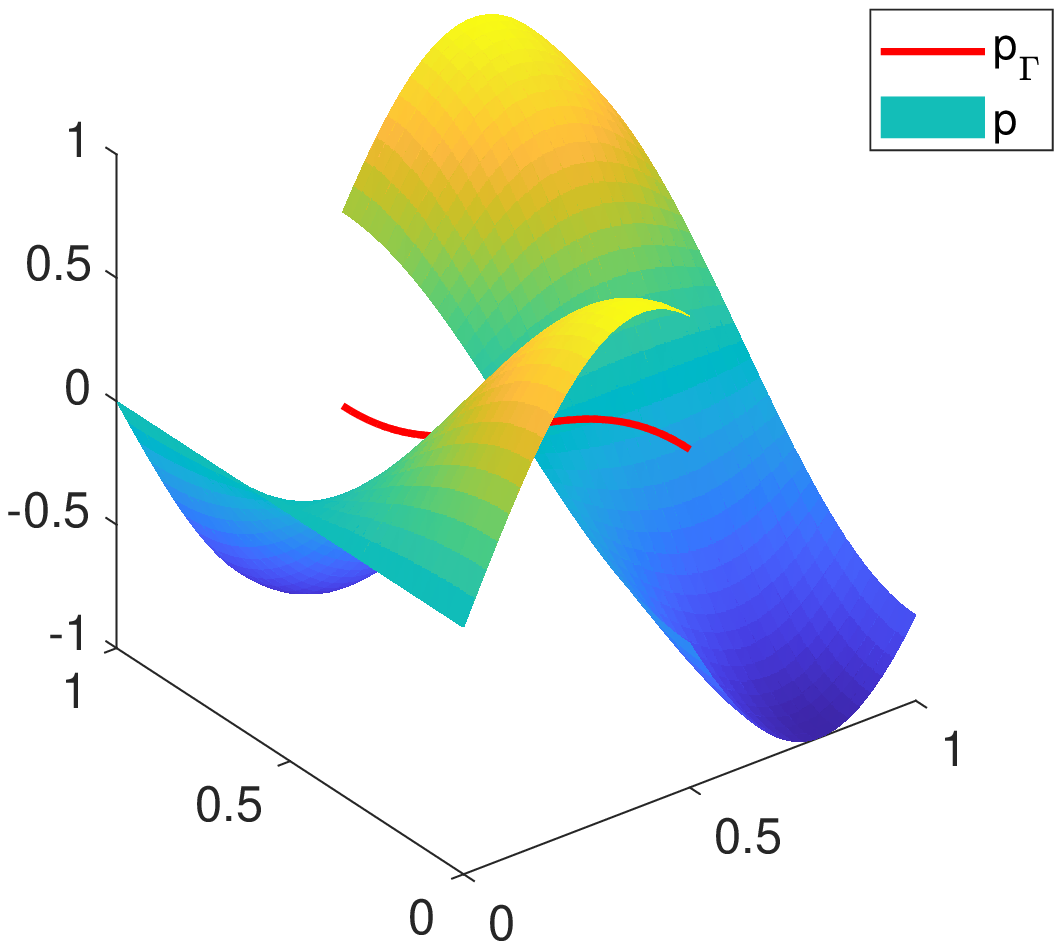}
    \includegraphics[width=0.45\textwidth]{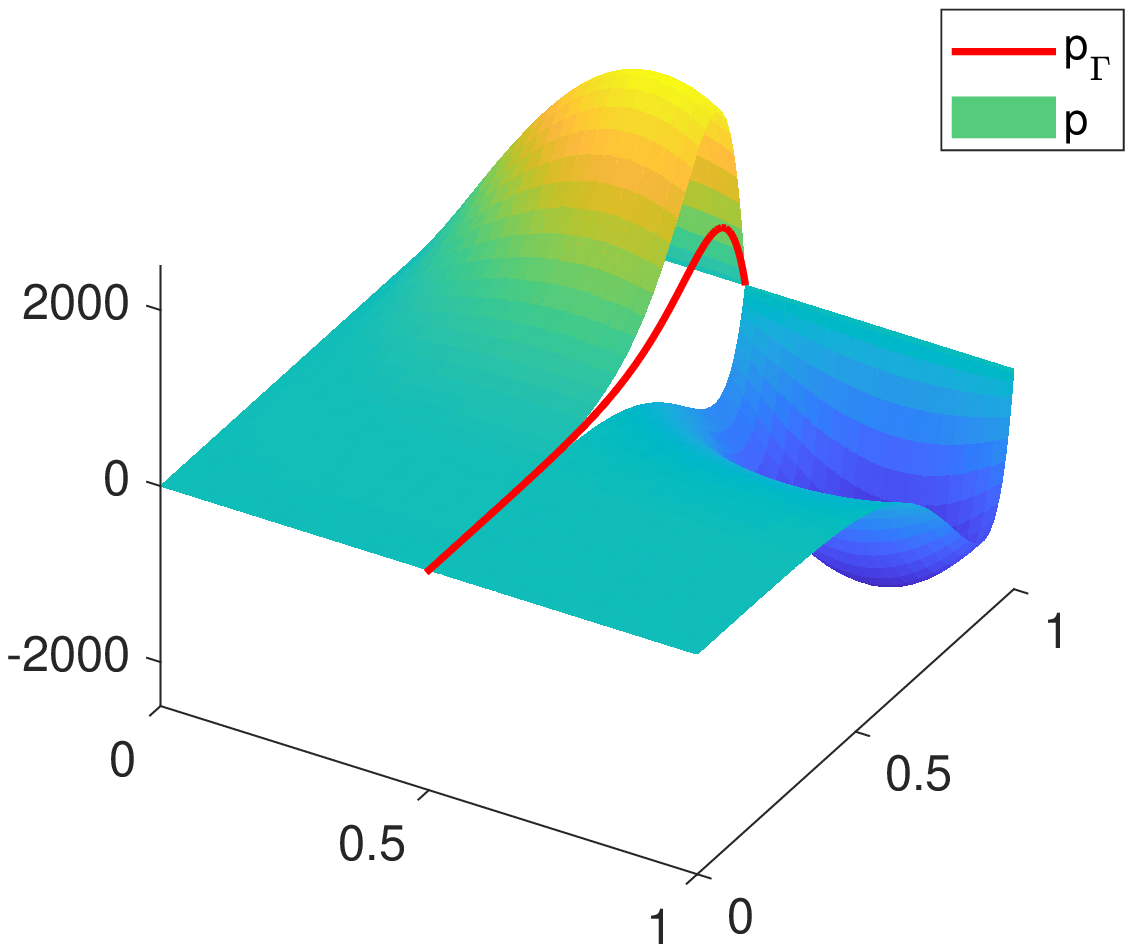}
    \caption{Graphs of solutions $p$ and $p_\Gamma$ for Example~\ref{subsec:example1} (left) and Example~\ref{subsec:example3} (right).}
    \label{fig:solshape}
\end{figure}

\begin{figure}[t]
    \centering
    \includegraphics[width=0.32\textwidth]{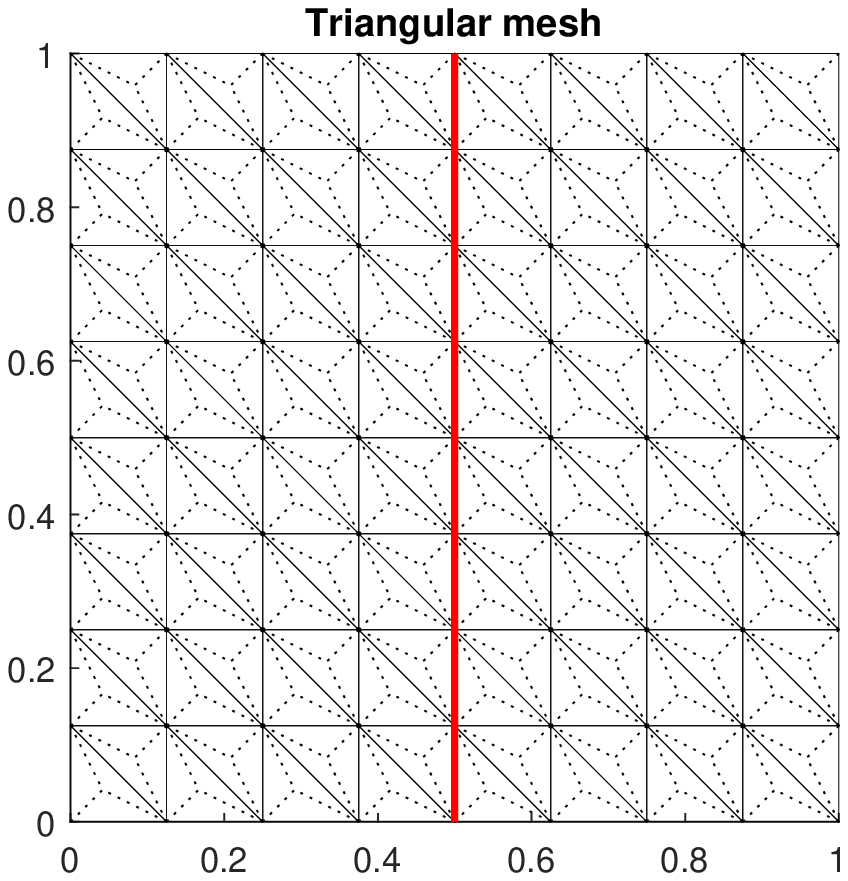}
    \includegraphics[width=0.32\textwidth]{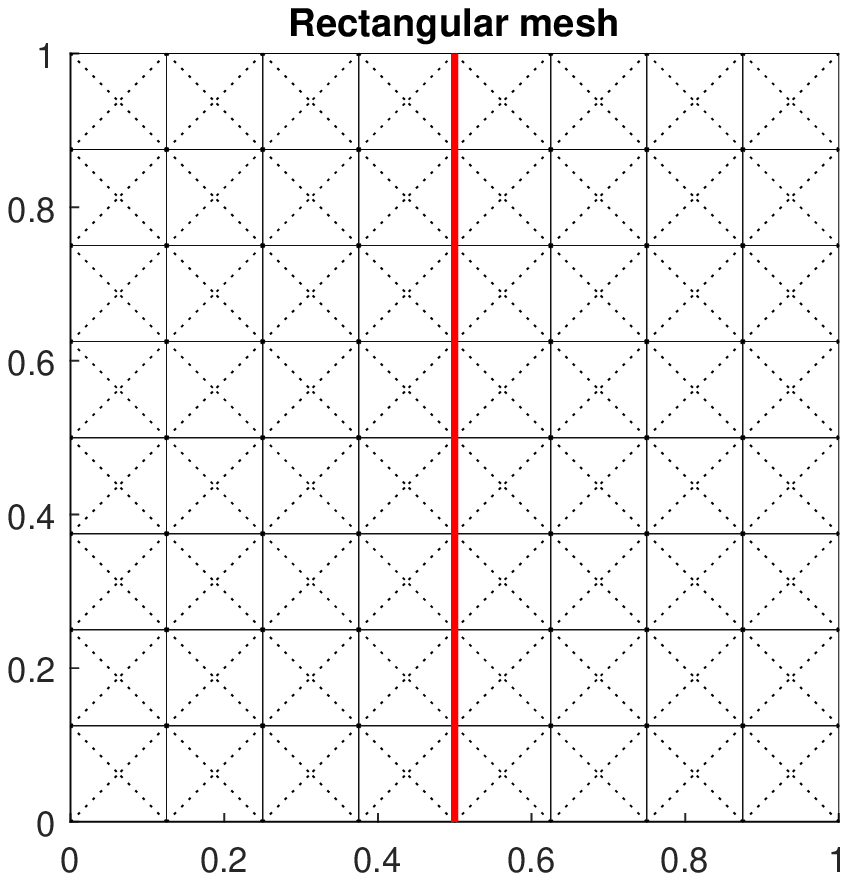}
    \includegraphics[width=0.32\textwidth]{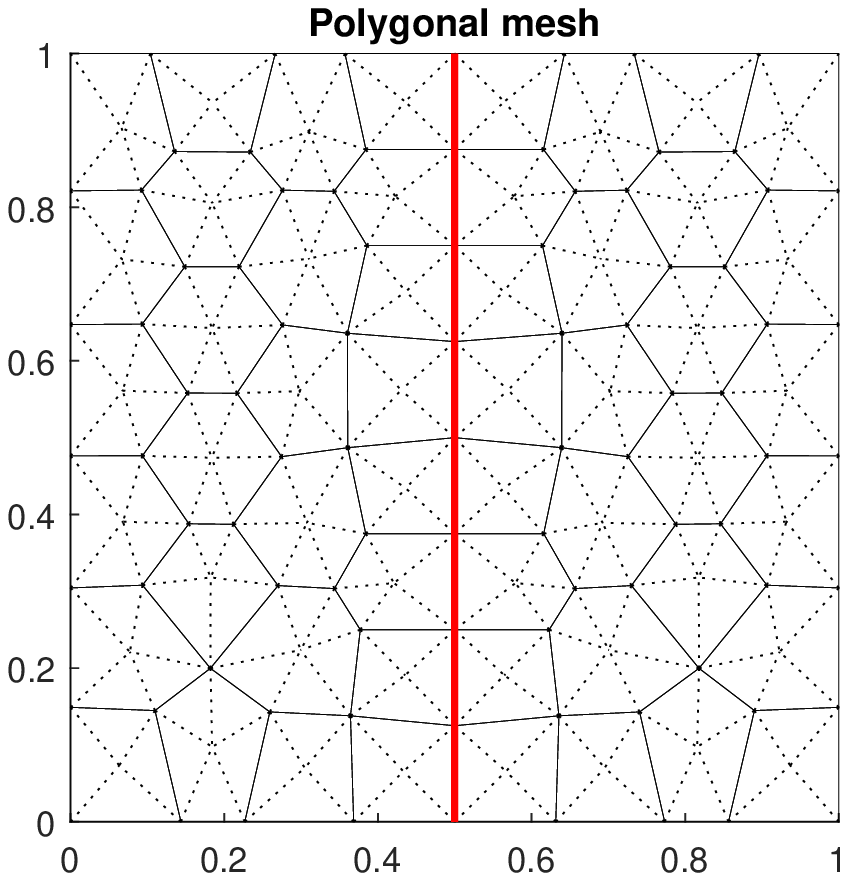}
    \caption{Uniform triangular (left), rectangular (center), polygonal (right) meshes with comparable mesh sizes for Example~\ref{subsec:example1}. Here, dashed lines represent dual edges and red lines are the fracture $\Gamma$.}
    \label{fig:mesh}
\end{figure}
\begin{figure}[t]
    \centering
    \includegraphics[width=0.32\textwidth]{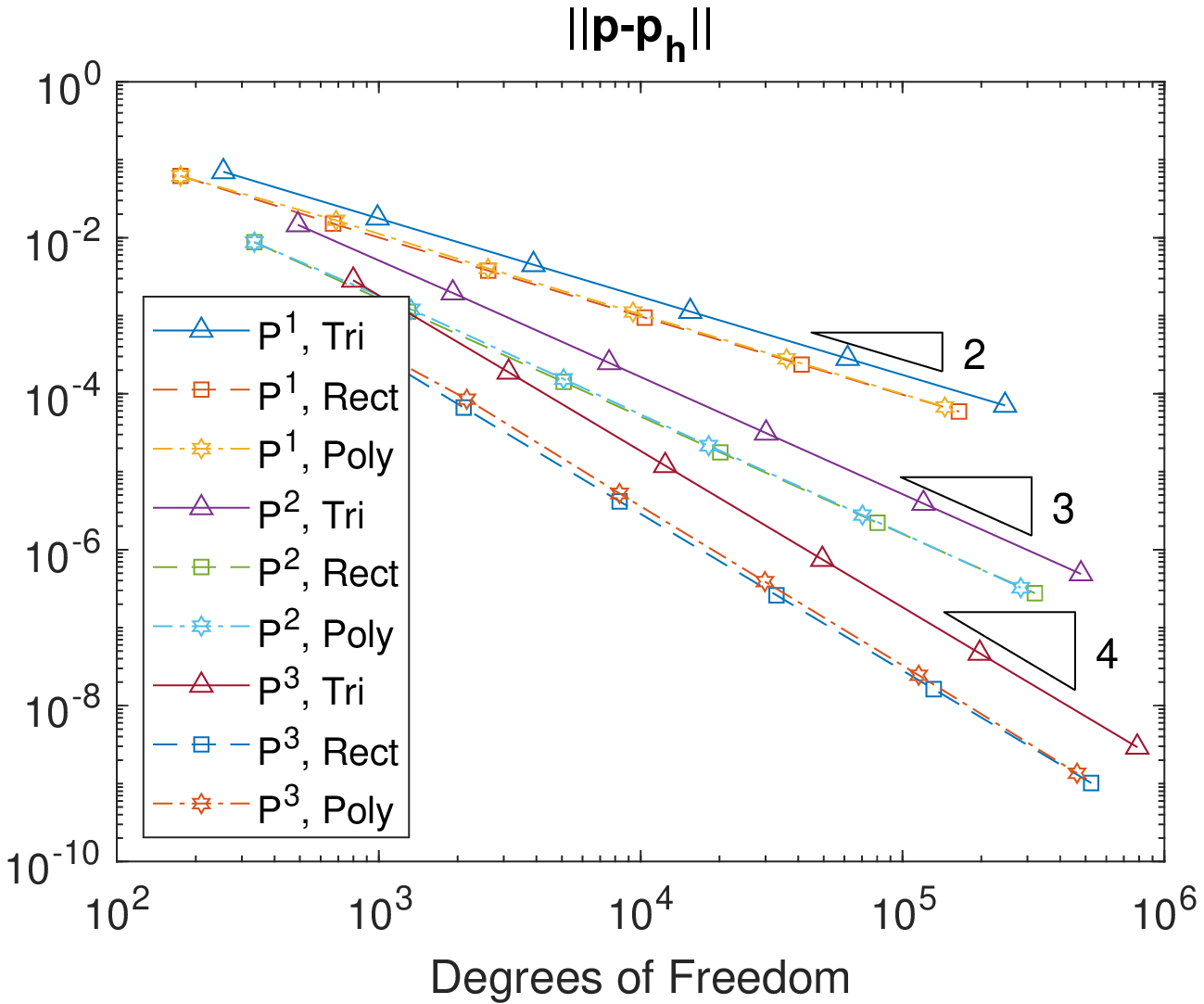}
    \includegraphics[width=0.32\textwidth]{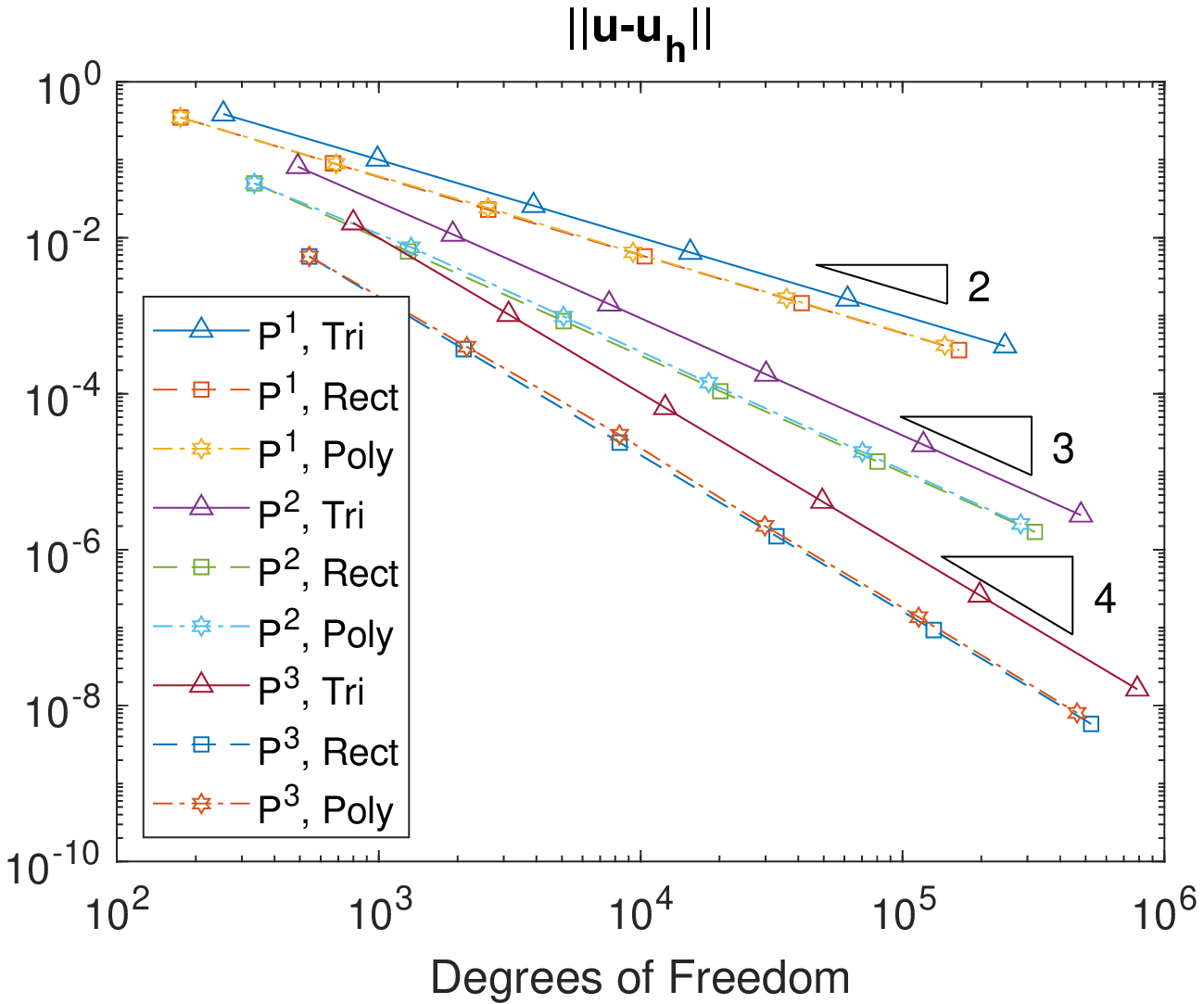}
    \includegraphics[width=0.32\textwidth]{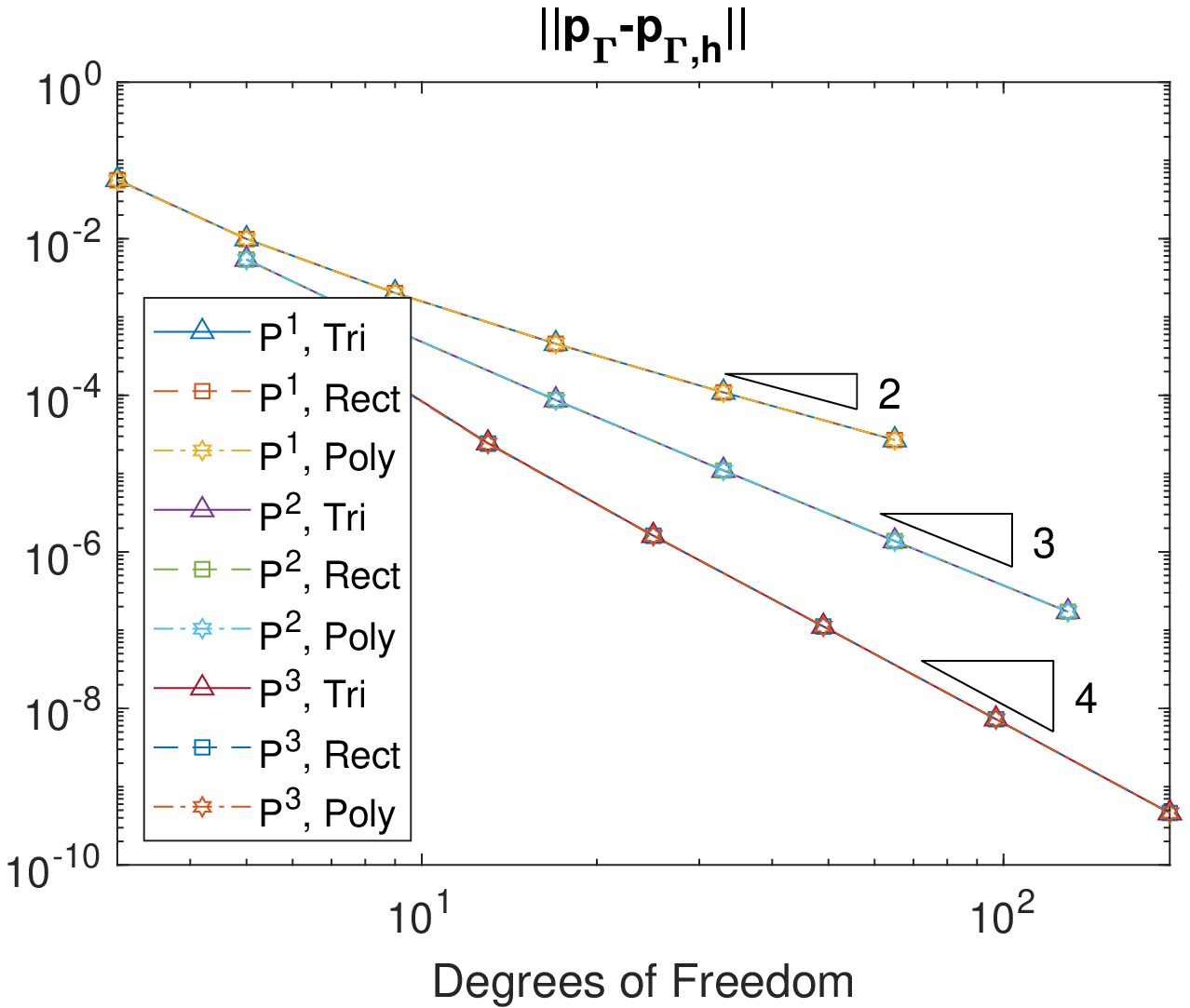}
    \includegraphics[width=0.32\textwidth]{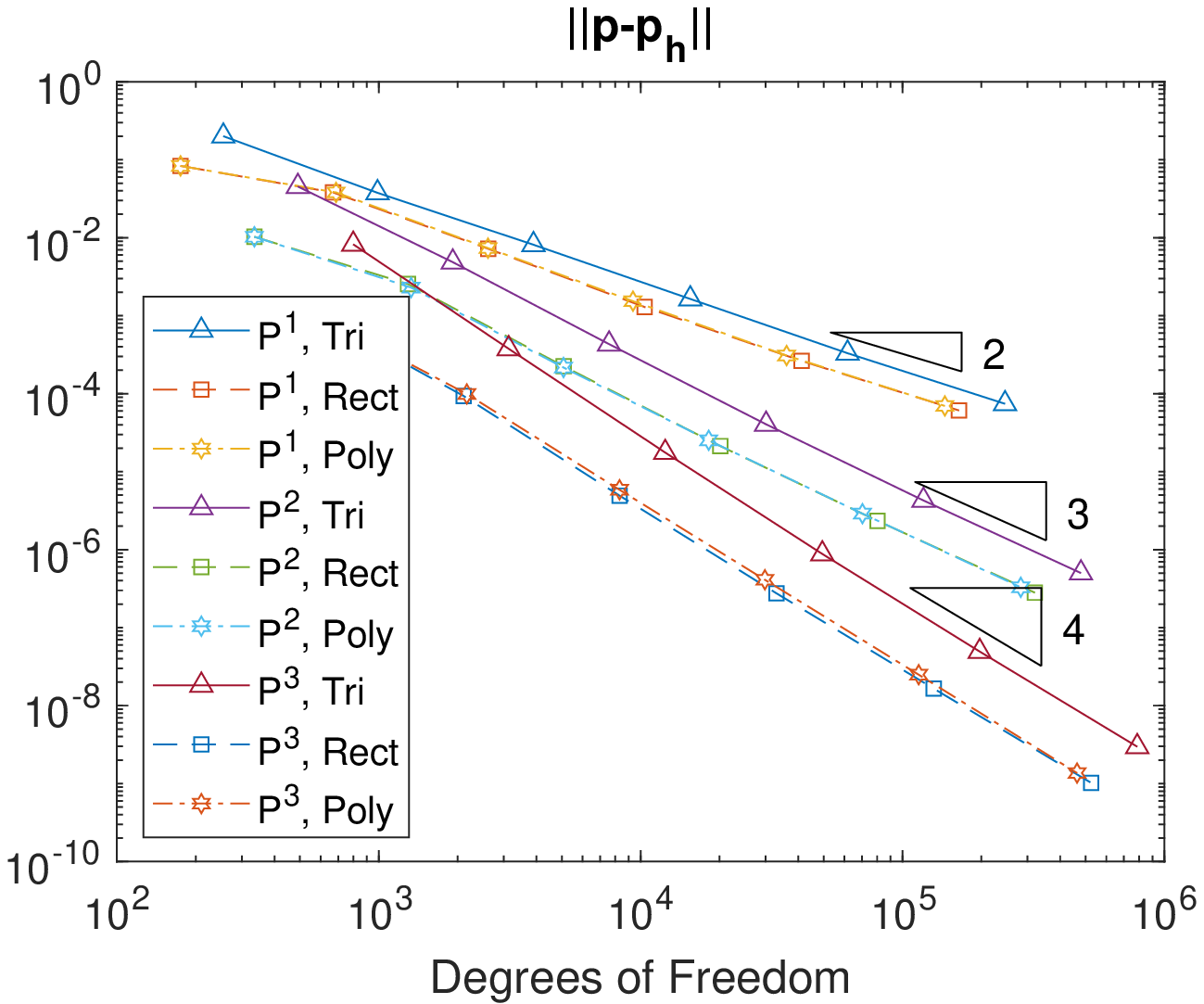}
    \includegraphics[width=0.32\textwidth]{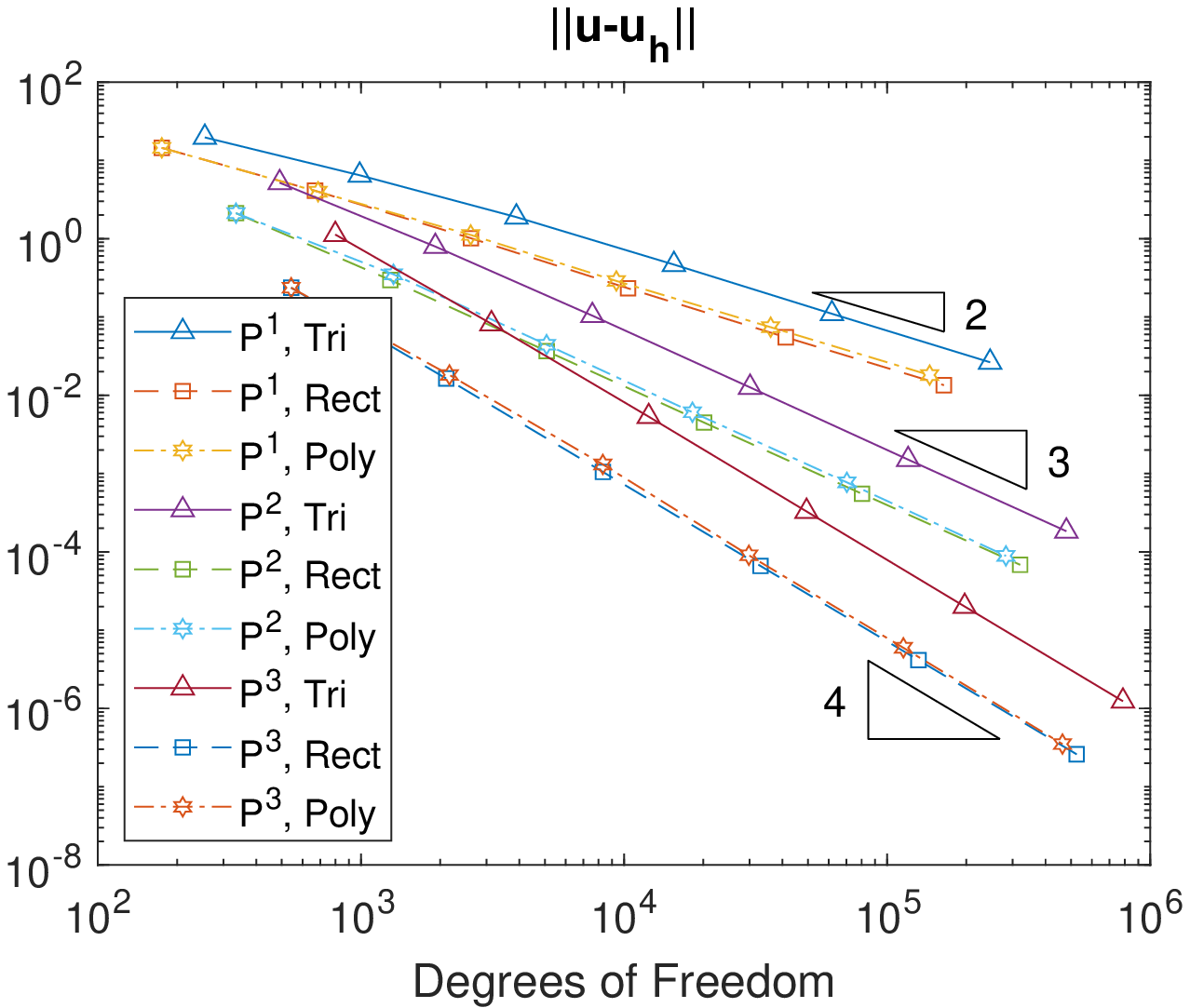}
    \includegraphics[width=0.32\textwidth]{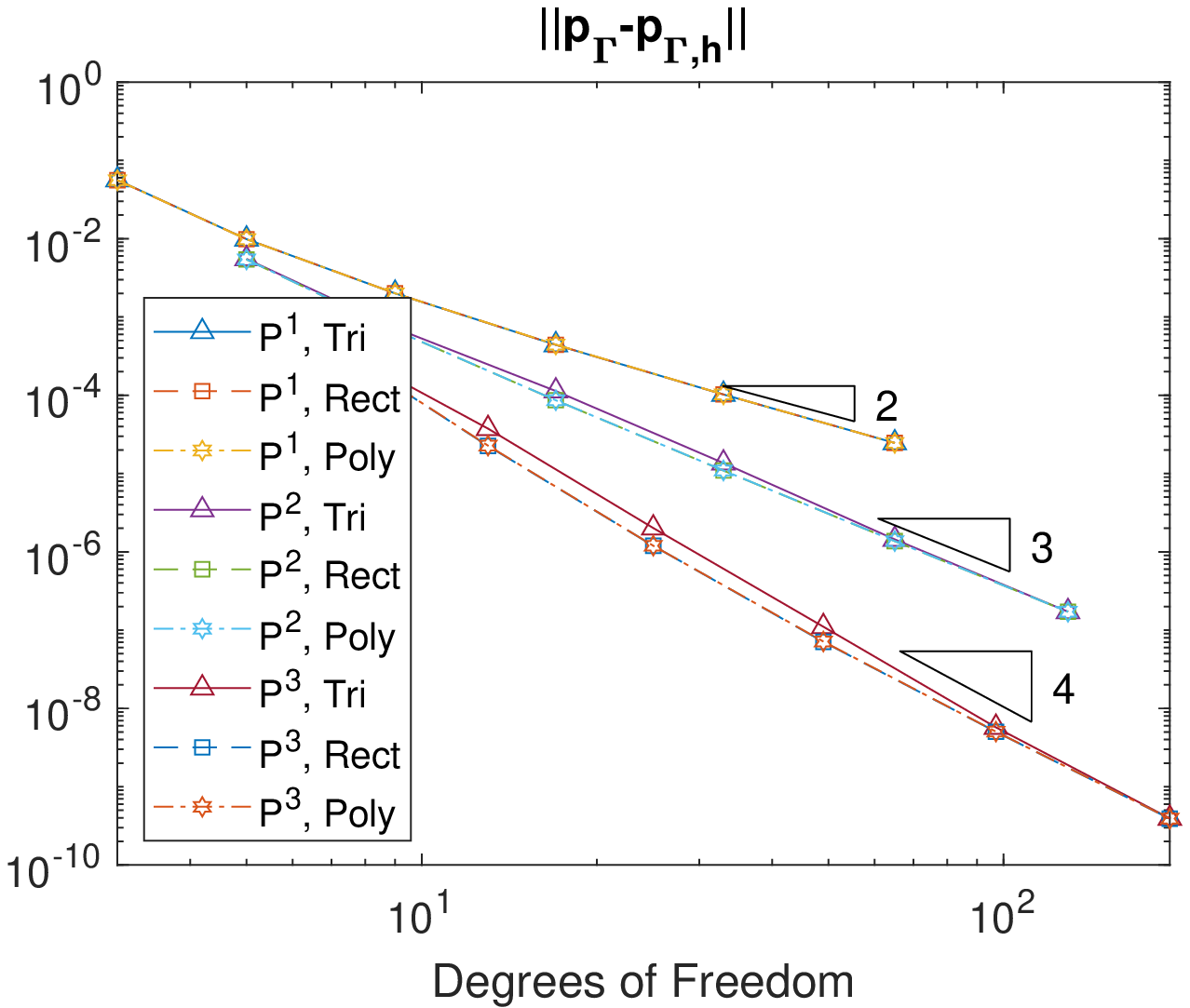}
    \caption{Convergence history for the isotropic case (top row) and the anisotropic case (bottom row) of Example~\ref{subsec:example1} with $k=1,2,3$.
        Right triangles indicate theoretical convergence rates.
        Solid lines, dotted lines, and dashed lines are error with triangular, rectangular, and polygonal meshes, respectively.}
    \label{fig:convhist}
\end{figure}
In this example, we verify the theoretical convergence order given in Corollary~\ref{cor:L2err}.
Let $\Omega=(0,1)\times(0,1)$ be the unit square in $\mathbb{R}^2$ with a fracture $\Gamma = \{x=0.5\}\times (0,1)$ at the middle.
Then, $\Omega_{B,1}=(0,0.5)\times(0,1)$ and $\Omega_{B,2}=(0.5,1)\times(0,1)$.
The solution $p$ and $p_\Gamma$ are defined by
\begin{equation*}
    p=\begin{cases}
        \sin(4x)\cos(\pi y) & \mbox{when }(x,y)\in\Omega_{B,1}, \\
        \cos(4x)\cos(\pi y) & \mbox{when }(x,y)\in\Omega_{B,2},
    \end{cases}\quad p_\Gamma = \frac{3}{4}\cos(\pi y)(\cos(2)+\sin(2)),
\end{equation*}
and the profiles for $p$ nd $p_G$ are depicted in Figure~\ref{fig:solshape}.
To demonstrate that our method handles anisotropic permeability, we consider two values of $\kappa_\Gamma^n$
\begin{equation*}
    \kappa_\Gamma^n=\begin{cases}0.01&\text{for isotropic case},\\1&\text{for anisotropic case}.\end{cases}
\end{equation*}
Other physical parameters are chosen as $\xi=3/4$, $\ell_\Gamma=0.01$, $K_\Gamma=1$ and
\begin{equation*}
    K = \left(
          \begin{array}{cc}
            \kappa_\Gamma^n/(2\ell_\Gamma) & 0\\
            0 & 1 \\
          \end{array}
        \right).
\end{equation*}
The boundary conditions and source terms can be derived from the definition.
Numerical tests are performed with three mesh types: Uniform triangular; uniform rectangular; quasi-uniform polygonal meshes, see Figure~\ref{fig:mesh}.
Quasi-uniform polygonal meshes are centroidal Voronoi tessellation (CVT) generated by the Lloyd algorithm with target mesh size $h$ (cf. \cite{Talischi12}).
We pre-allocate fixed generators for Voronoi cells near the fracture so that the resulting mesh aligns with the fracture $\Gamma$.
In Figure~\ref{fig:convhist}, we report the convergence history for $\|p-p_h\|_{0,\Omega_B}, \|\bm{u}-\bm{u}_h\|_{0,\Omega_B}$ and $\|p_\Gamma-p_{\Gamma,h}\|_{0,\Gamma}$ against the number of degrees of freedom for polynomial order $k=1,2,3$ for both isotropic and anisotropic cases. We can observe optimal convergence rates $\mathcal{O}(h^{k+1})$ regardless of the choice of $K$, which confirms the theoretical findings.
In addition, the polygonal meshes and rectangular meshes outperform triangular meshes in terms of accuracy.
This and the geometrical flexibility of polygonal meshes reveal that polygonal meshes are better suited to the simulation of physical problems under consideration.

\subsection{Robustness to small edges}\label{subsec:example2}
\begin{figure}[t]
    \centering
    \includegraphics[width=0.32\textwidth]{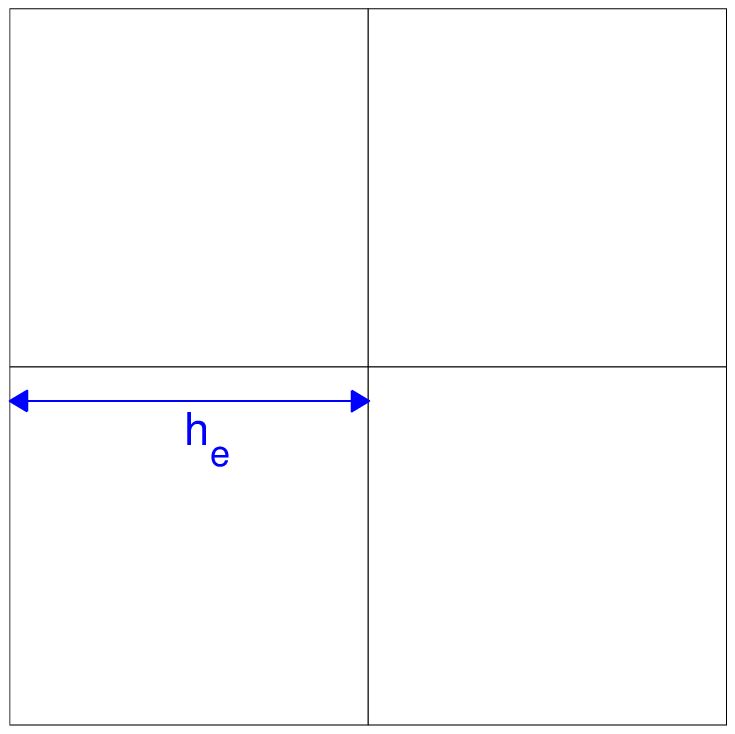}
    \includegraphics[width=0.32\textwidth]{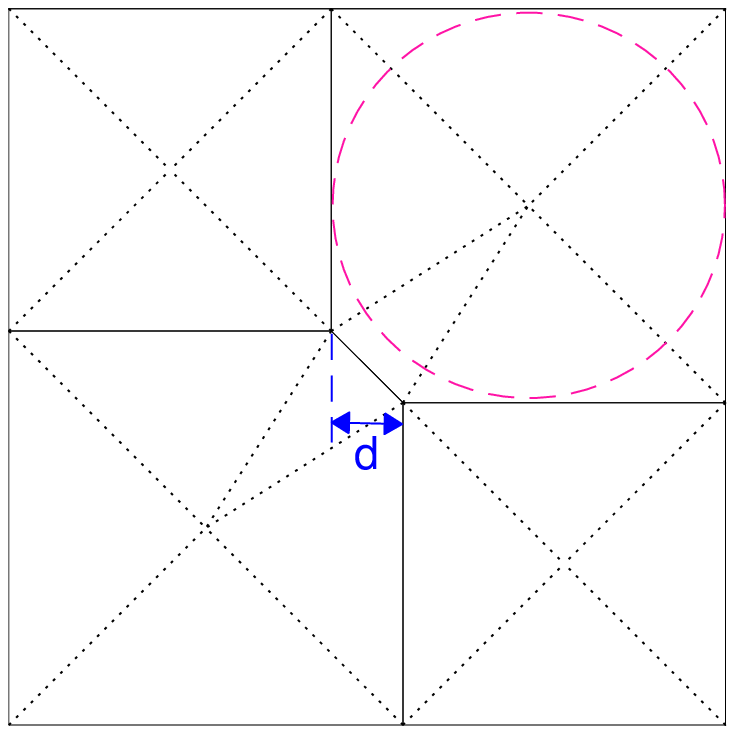}
    \includegraphics[width=0.32\textwidth]{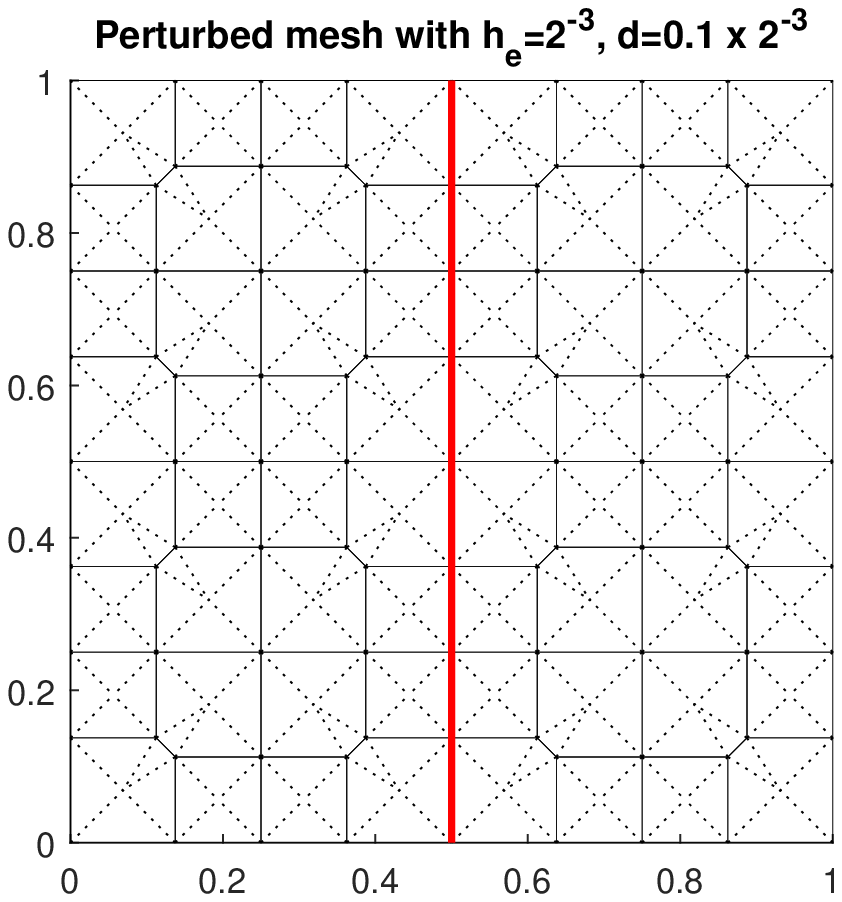}
    \caption{Schematic of perturbation. $2\times2$ squares (left), two rectangles and two pentagons after perturbation with $d=0.1\times h_e$ (center), and a resulting mesh from a uniform rectangular mesh with $h_e=2^{-3}$ and $d=0.1\times h_e$.
    The dashed circle is the ball, described in Assumption (A), of an pentagon.}
    \label{fig:mesh_smallEdge}
\end{figure}

\begin{figure}[t]
    \centering
    \includegraphics[width=0.32\textwidth]{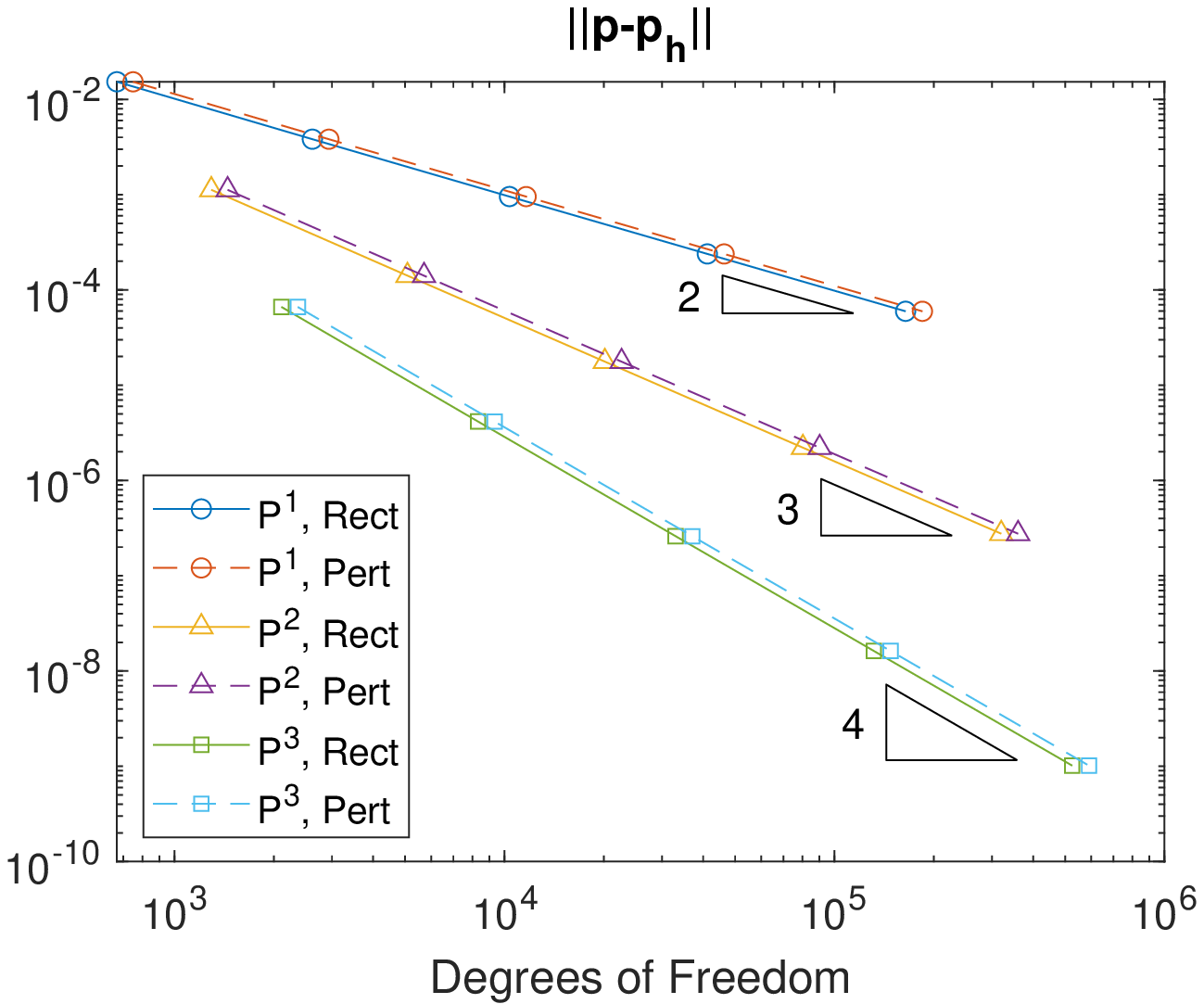}
    \includegraphics[width=0.32\textwidth]{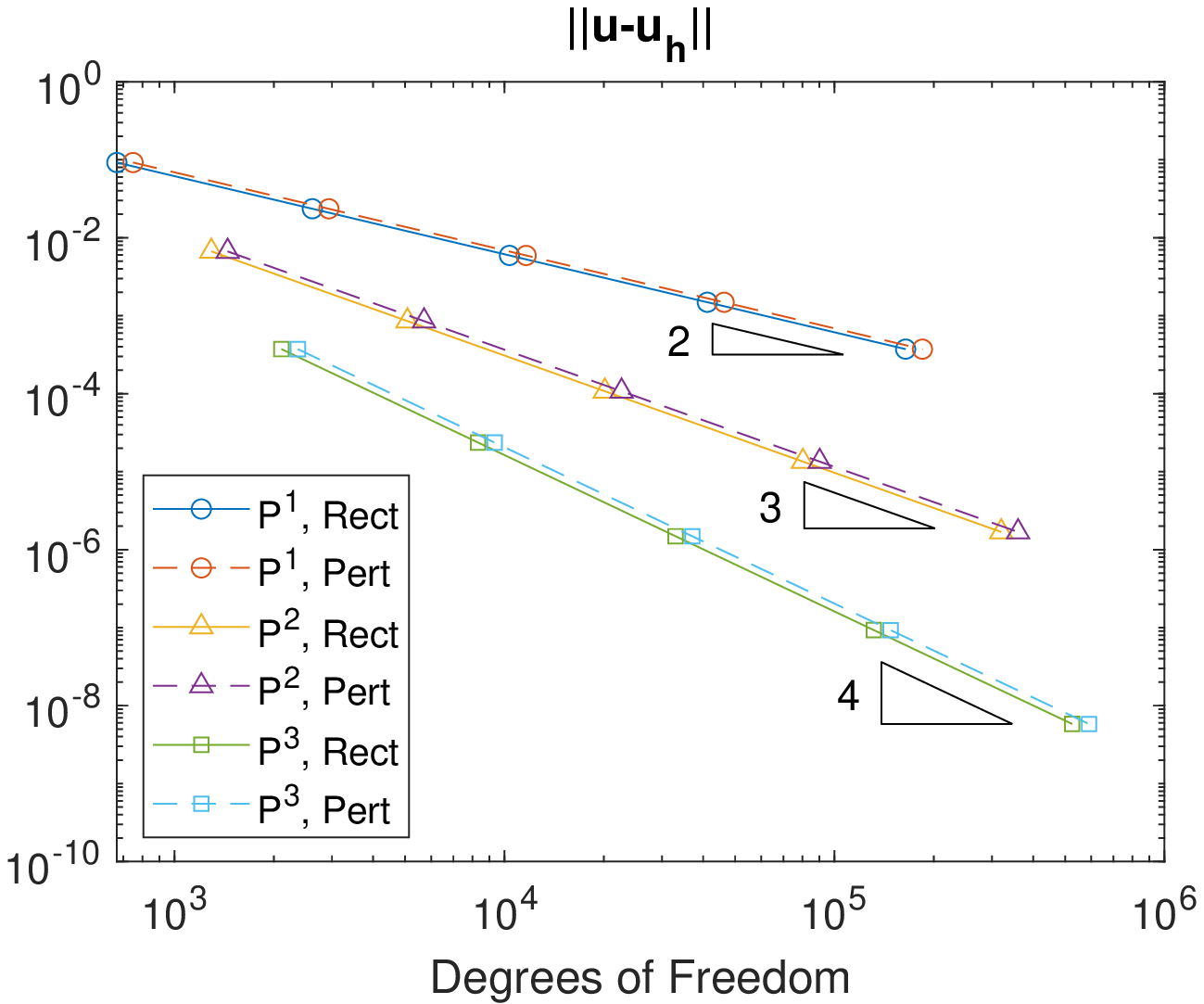}
    \includegraphics[width=0.32\textwidth]{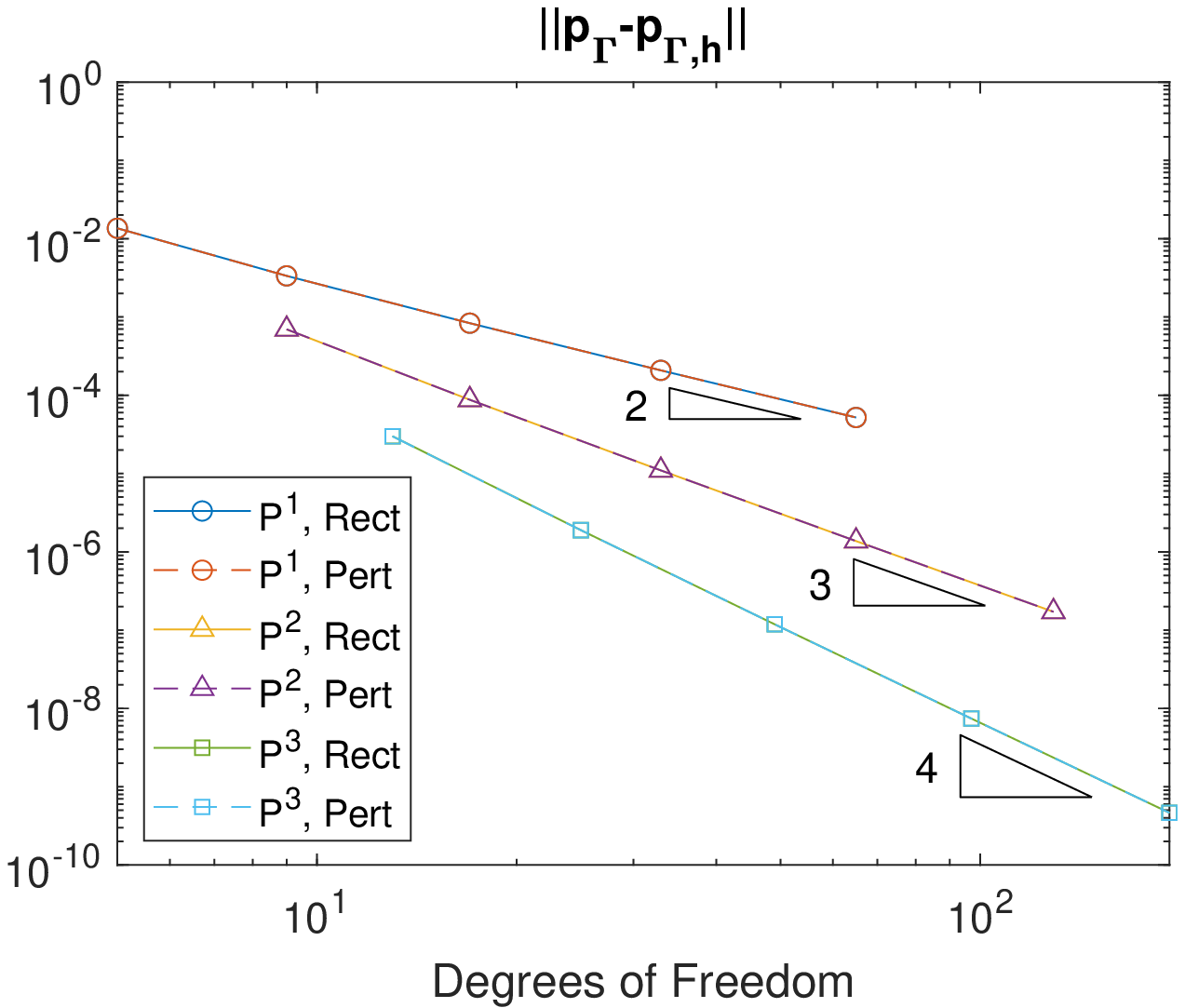}
    \caption{Convergence history with uniform rectangular meshes (solid lines) and perturbed meshes with $d=0.001\times h_e$ (dashed lines)}
    \label{fig:convhist_smallEdge}
\end{figure}
In Corollary~\ref{cor:L2err}, \textit{a priori} error estimates for $||u-u_h||_{0,\Omega_B}$, $||p-p_h||_{0,\Omega_B}$ and $||p_\Gamma-p_{\Gamma,h}||_{0,\Gamma}$ are derived without Assumption (B).
Therefore, the accuracy of all three variables in $L^2$ error should be independent of the existence of small edges.
To demonstrate this, we design a mesh by perturbing a uniform rectangular mesh.
For each $2\times2$ squares, we replace the common vertex of four squares by a small edge with length $\sqrt{2}d$ so that we obtain two squares and two pentagons, see Figure~\ref{fig:mesh_smallEdge}.
By taking $d$ small enough, the differences between the mesh sizes of uniform rectangular mesh and its perturbed mesh are negligible.
While the perturbation preserves $\rho_S$ in Assumption (A), $\rho_E$ in Assumption (B) is changed from $1/\sqrt{2}$ to $d/h_e$ for the perturbed mesh.

For numerical tests, we consider the same solution used in Example~\ref{subsec:example1}.
The convergence history against the number of degrees of freedom are reported in Figure~\ref{fig:convhist_smallEdge} for polynomial order $k=1,2,3$.
Here, we use the uniform rectangular mesh and perturbed mesh with $d=0.001\times h_e$.
We can observe that the difference in accuracy between the numerical solutions with uniform rectangular meshes and perturbed meshes for all the three variables are negligible.
Also, noting that the perturbation introduces additional degrees of freedom due to the small edges, the difference can be attributed to the increased degrees of freedom.
Therefore, we conclude that our method is free from Assumption (B).

\subsection{Anisotropic meshes}\label{subsec:example3}
\begin{figure}[t]
    \centering
    \includegraphics[width=0.32\textwidth]{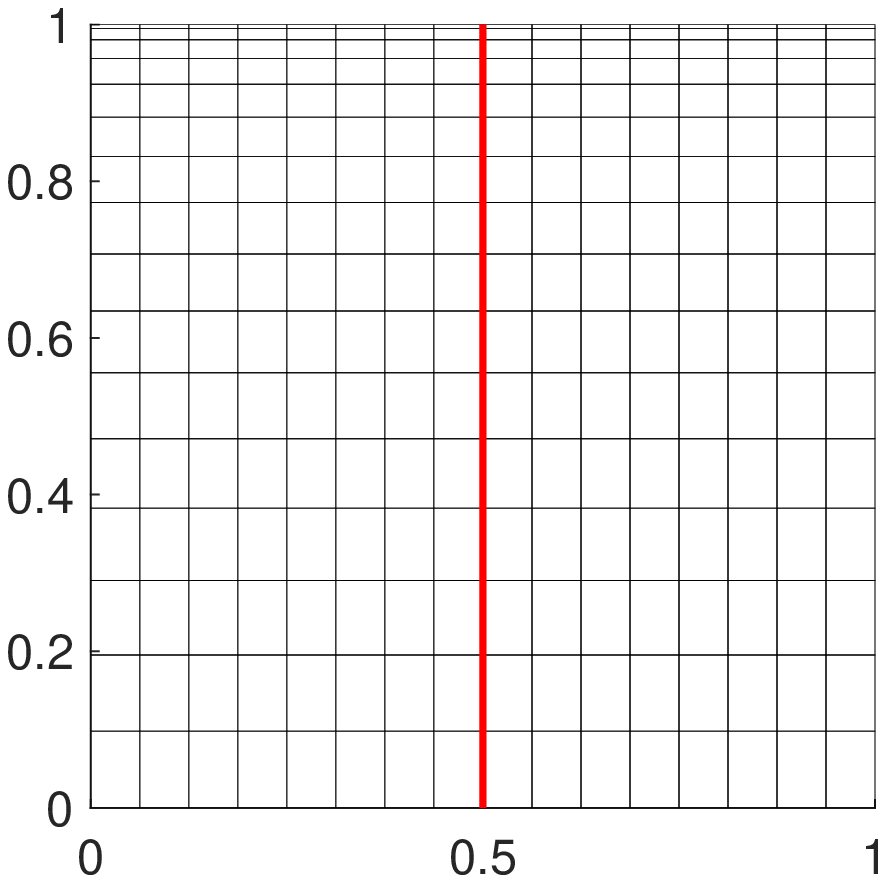}
    \includegraphics[width=0.32\textwidth]{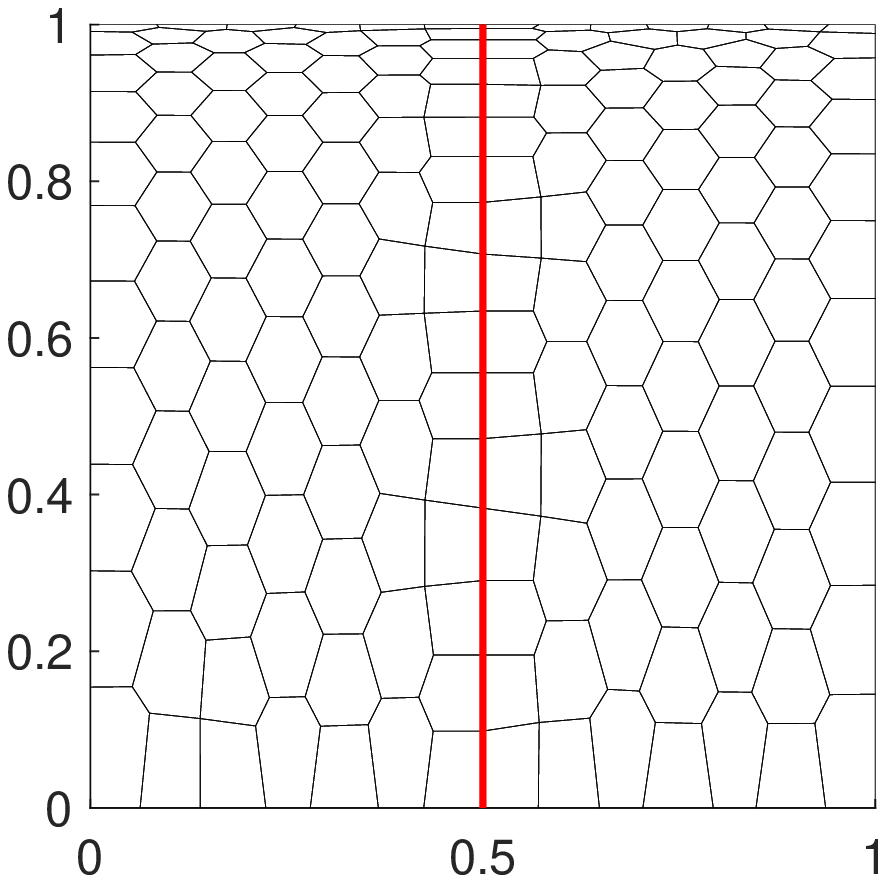}
    \caption{Mapped rectangular (left) and polygonal (right) meshes from (quasi-)uniform meshes used in Example~\ref{subsec:example3}. Polygons near $y=1$ are highly anisotropic so that the Assumption~(A) is not satisfied.}
\end{figure}
\begin{figure}[t]
    \centering
    \includegraphics[width=0.32\textwidth]{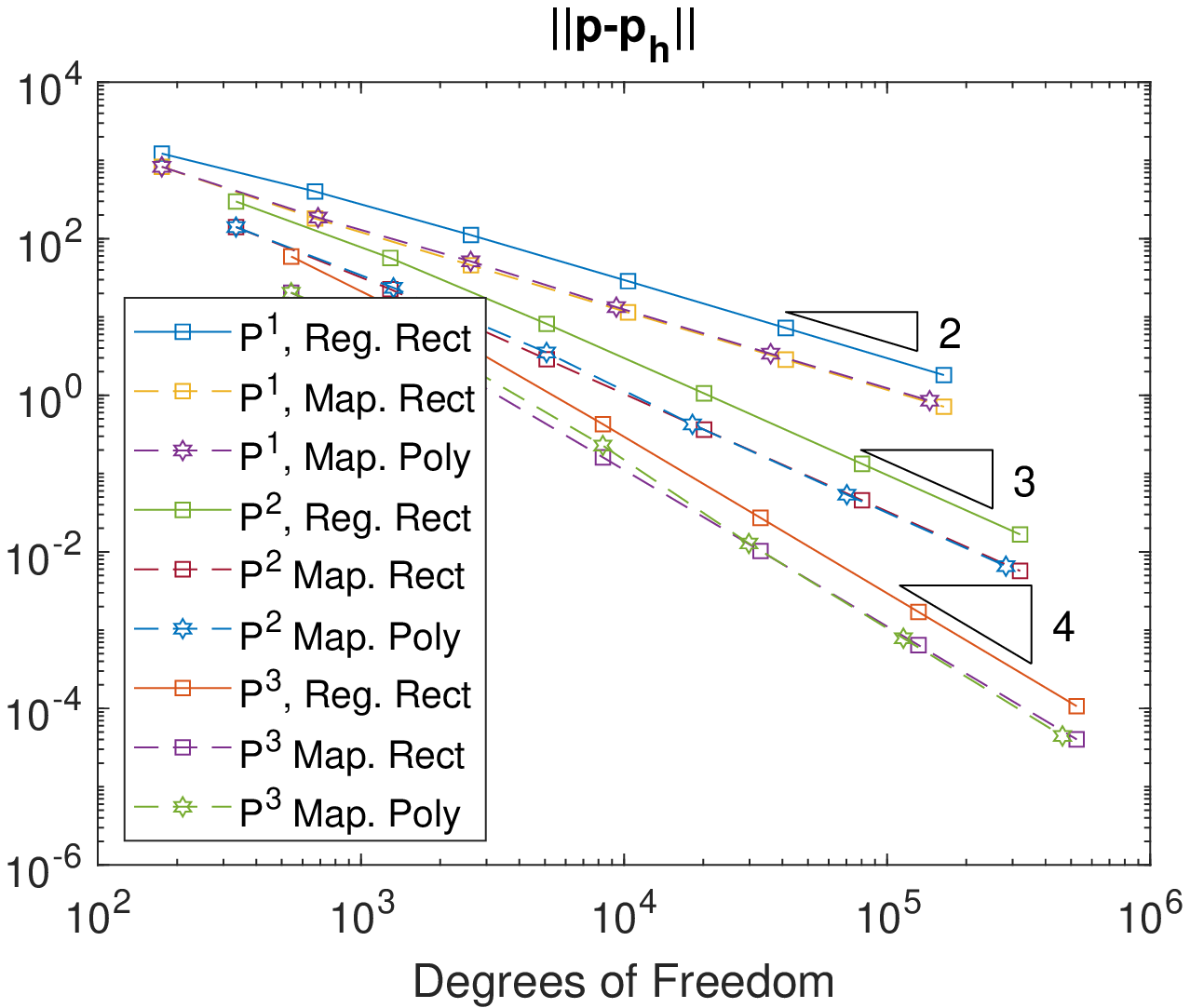}
    \includegraphics[width=0.32\textwidth]{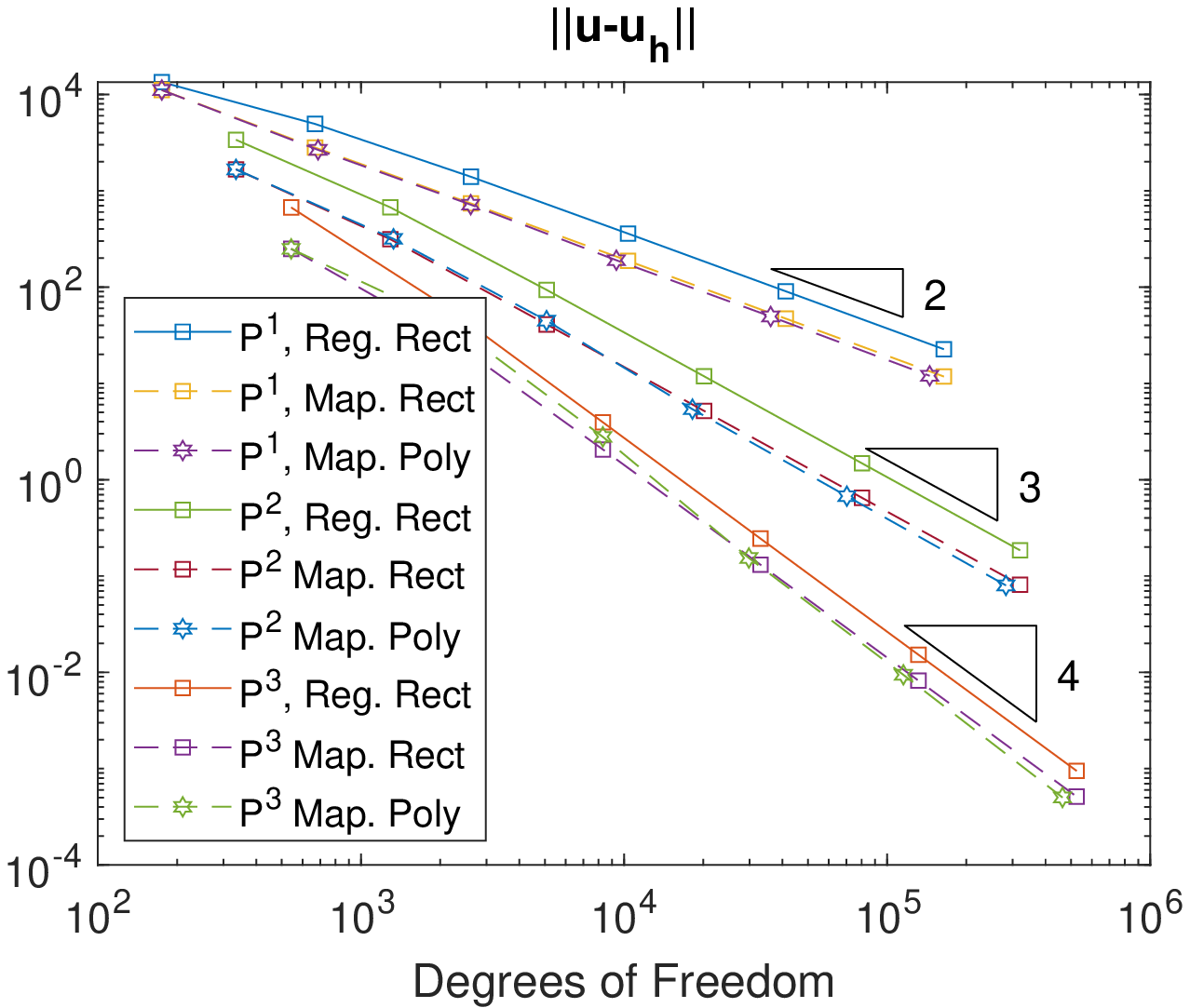}
    \includegraphics[width=0.32\textwidth]{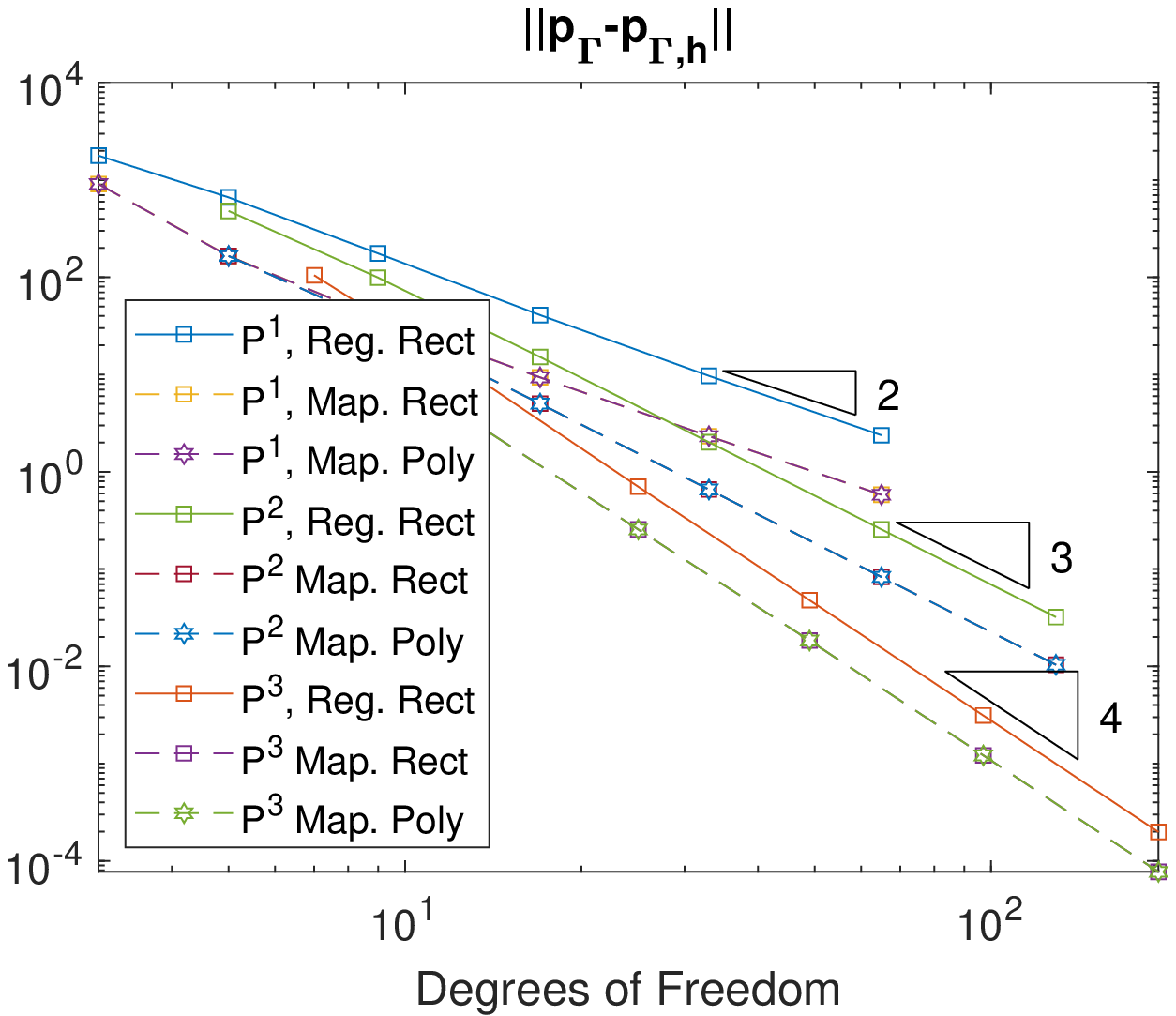}
    \caption{Convergence history with uniform rectangular meshes (solid lines) and anisotropic meshes (dashed lines) for $k=1,2,3$. Reg. and Map. in legends are abbreviation of regular and mapped meshes, respectively.}
    \label{fig:convhist_anisotropic}
\end{figure}
In this example, we investigate reliability of the proposed method when it is used with anisotropic meshes.
Consider the solution $p$ and $p_\Gamma$ defined by
\begin{equation*}
    p=\begin{cases}
        \exp(10y)\sin(4x)\sin(\pi y) & \mbox{when }(x,y)\in\Omega_{B,1}, \\
        \exp(10y)\cos(4x)\sin(\pi y) & \mbox{when }(x,y)\in\Omega_{B,2},
    \end{cases}\quad p_\Gamma = \frac{3}{4}\exp(10y)\sin(\pi y)(\cos(2)+\sin(2)).
\end{equation*}
The domain $\Omega$, fracture $\Gamma$, and other physical constants are chosen as for the isotropic case of Example~\ref{subsec:example1}.
Notice that the solutions $p$ and $p_\Gamma$ exhibit boundary layer on $y=1$, see Figure~\ref{fig:solshape}.

Consider a mapping $A:(0,1)\times(0,1)\rightarrow(0,1)\times(0,1)$ by $(x,y)\mapsto (x,\sin(\pi y/2))$.
This maps a regular mesh on $(0,1)\times(0,1)$ to a highly anisotropic mesh near $y=1$.
The resulting mapped mesh violates Assumption (A) since the radius of the ball converges to 0 near $y=1$ as $h$ converges to 0.
Again, we report the convergence history against the number of degrees of freedom on mapped rectangular and polygonal meshes for $k=1,2,3$, see Figure~\ref{fig:convhist_anisotropic}.
For a reference, we also include numerical results with uniform rectangular meshes.
We can observe optimal convergence for all the three variables even when highly anisotropic meshes are employed.
Also, due to densely distributed mesh near the boundary layer, the anisotropic meshes give more accurate results than uniform meshes.
Moreover, the loss of accuracy can be also attributed to the non-uniform distribution of elements.

\subsection{Unfitted general meshes}\label{subsec:example4}

\begin{figure}[t]
    \centering
    \includegraphics[width=0.32\textwidth]{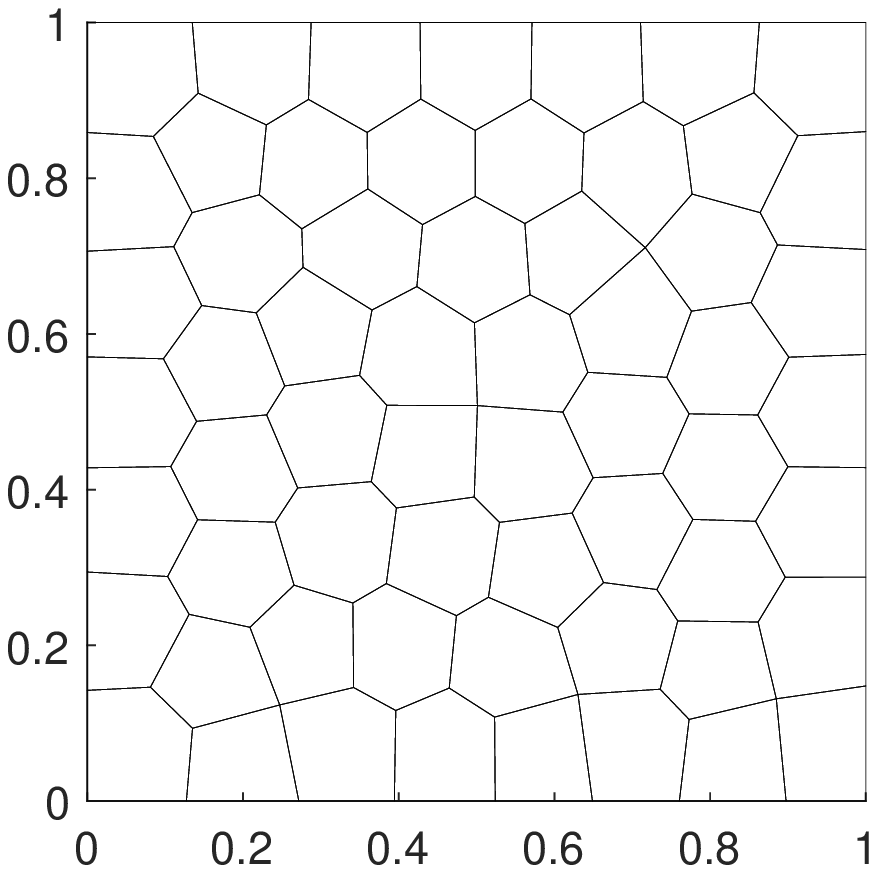}
    \includegraphics[width=0.32\textwidth]{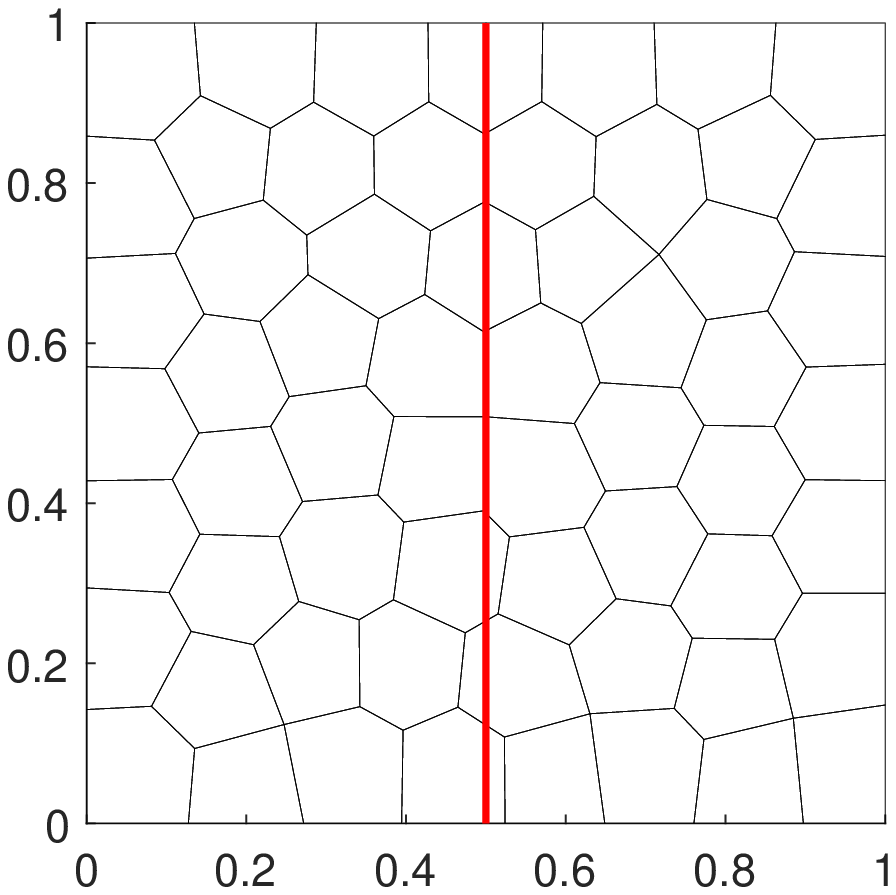}
    \includegraphics[width=0.32\textwidth]{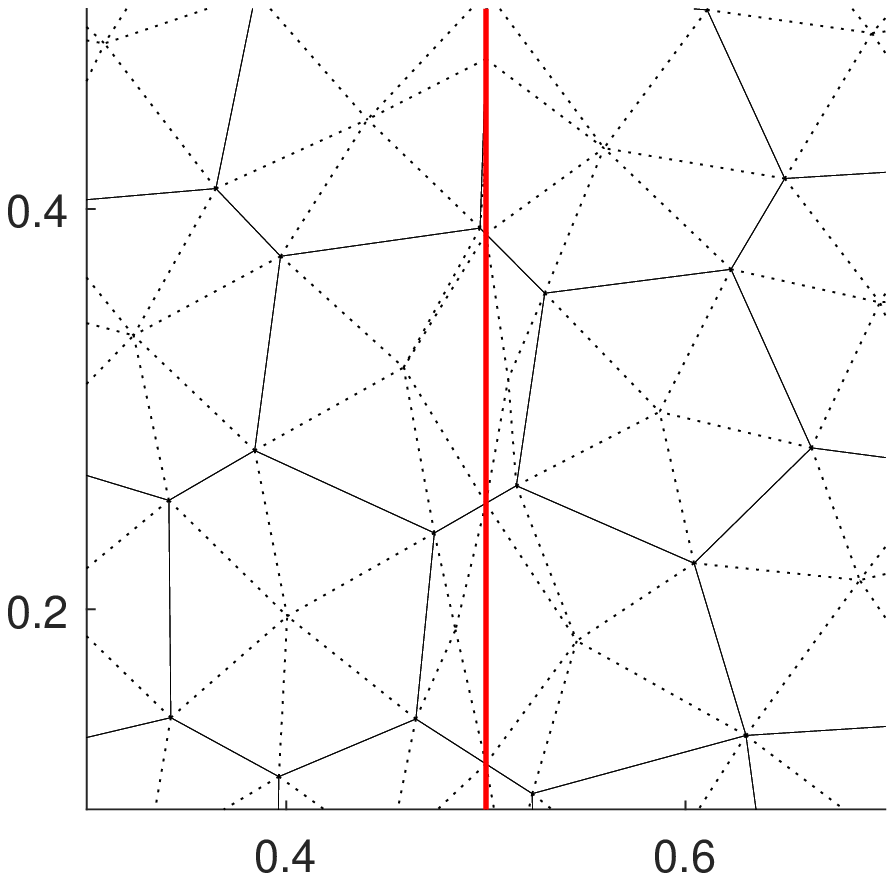}
    \caption{Underlying polygonal mesh ($\mathcal{T}_u$, left), modified mesh ($\tilde{\mathcal{T}}_u$) (center) and its magnified view with dual edges (right). The modified mesh contains both sliver elements and small edges.}
    \label{fig:mesh_unfitted}
\end{figure}

\begin{figure}[t]
    \centering
    \includegraphics[width=0.32\textwidth]{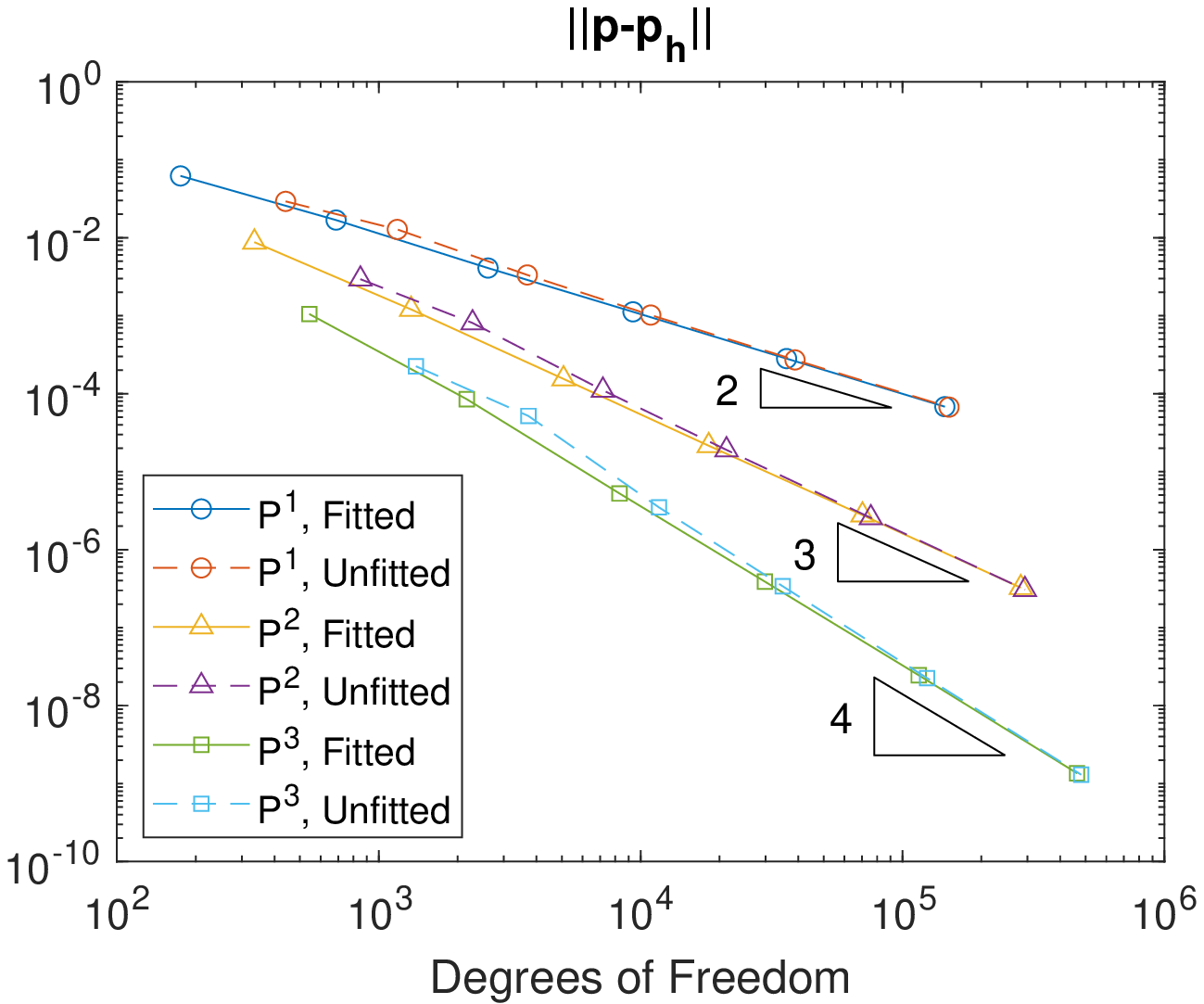}
    \includegraphics[width=0.32\textwidth]{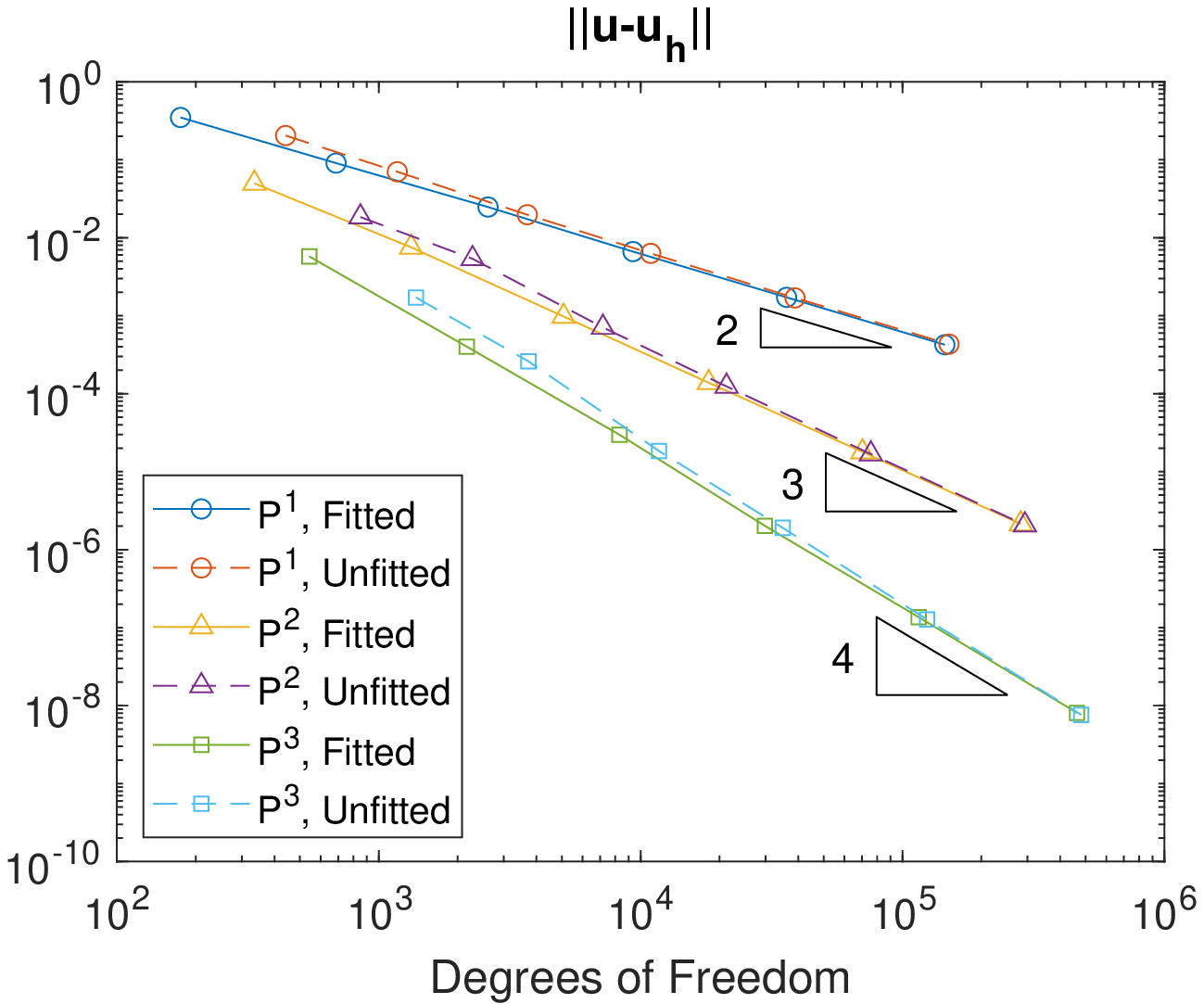}
    \includegraphics[width=0.32\textwidth]{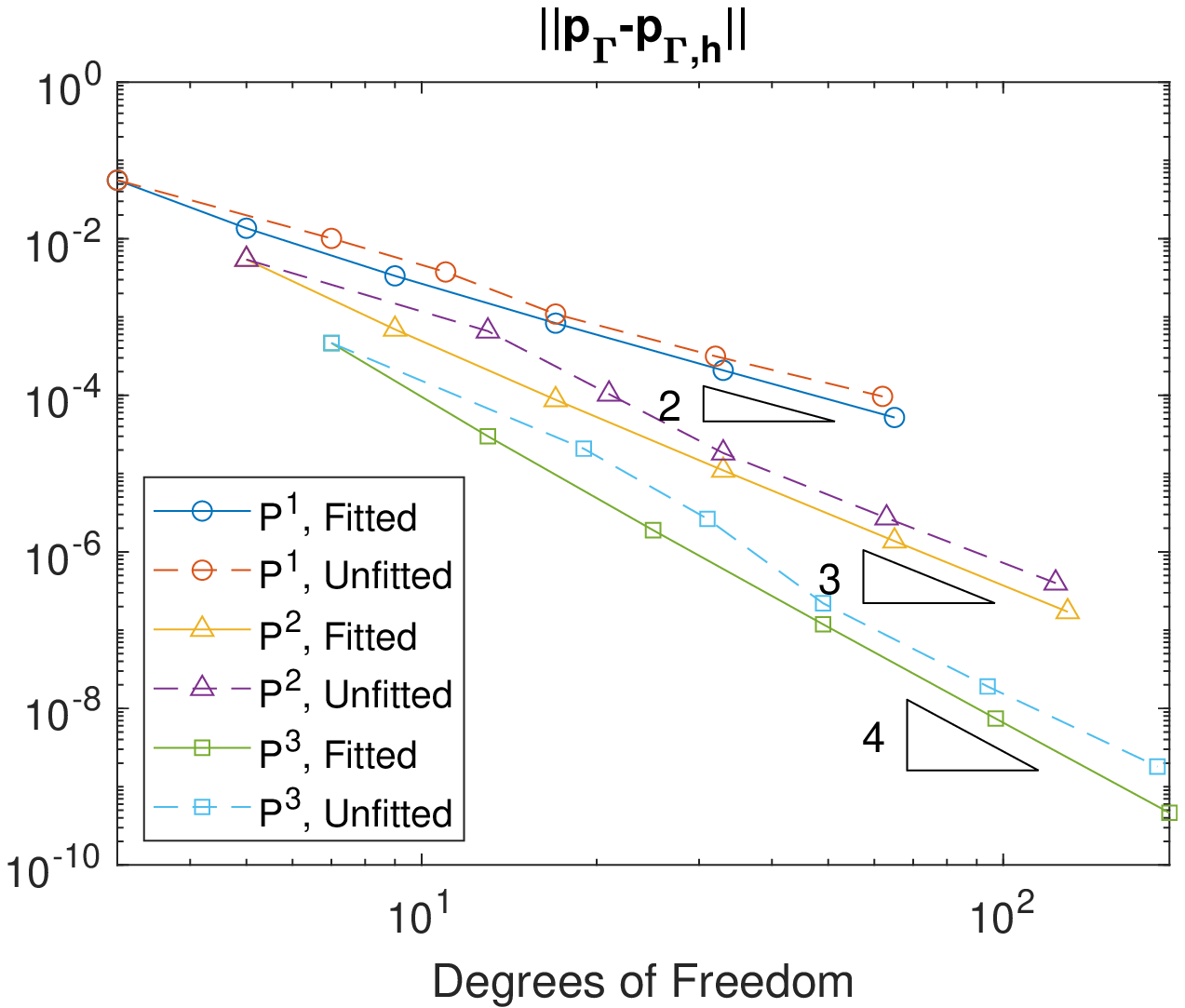}
    \caption{Convergence history with fitted (solid lines) and unfitted (dashed lines).}
    \label{fig:convhist_unfitted}
\end{figure}
In Example~\ref{subsec:example1}, we generate fitted polygonal meshes with pre-fixed generators near $\Gamma$.
However, this method is not applicable in practice because of complex, or even unknown \text{a priori}, geometry of the fracture.
Therefore, numerical methods utilizing unfitted meshes are preferred.
Example~\ref{subsec:example2} and Example~\ref{subsec:example3} suggest that the proposed method is reliable and accurate even when polygonal meshes with small edges or sliver elements are employed.
This allow us to consider unfitted meshes for our method without additional treatment such as mesh aggregation \cite{badia2018} or removal of small edges.

Let $\mathcal{T}_u$ be a polygonal mesh on $\Omega$ which is generated independent of $\Gamma$.
For each polygon $T\in \mathcal{T}_u$ with $T\cap \Gamma\neq \emptyset$, we split $T$ into $\{T_i\}$ so that $T_i\subset \Omega_{B,i}$ for each $i$.
This induces a new mesh $\tilde{\mathcal{T}}_u$ where there is one and only one $\Omega_{B,i}$ for each $T\in\tilde{\mathcal{T}}_u$ such that $T\subset \Omega_{B,i}$.
Figure~\ref{fig:mesh_unfitted} shows an example of background mesh $\mathcal{T}_u$ and updated mesh $\tilde{\mathcal{T}}_u$ with its induced simplicial sub-meshes $\mathcal{T}_h$.
The convergence history with unfitted mesh is depicted in Figure~\ref{fig:convhist_unfitted}.
We also report convergence history with quasi-uniform polygonal meshes which is generated with pre-fixed generators as in Figure~\ref{fig:mesh}.
As expected from the observation made in Example~\ref{subsec:example2} and \ref{subsec:example3}, the proposed method gives optimally convergent numerical approximations for all the three variables.
Moreover, the accuracy of numerical approximations with unfitted meshes are similar to that with fitted meshes for bulk variables $p_h$ and $u_h$.
While the accuracy of $p_{\Gamma,h}$ with unfitted meshes is slightly lower than that with fitted meshes, considering the flexibility of the method, the difference is moderate.

\subsection{Quarter five-spot problem}\label{subsec:qfs}
\begin{figure}[t]
    \centering
    \includegraphics[width=0.45\textwidth]{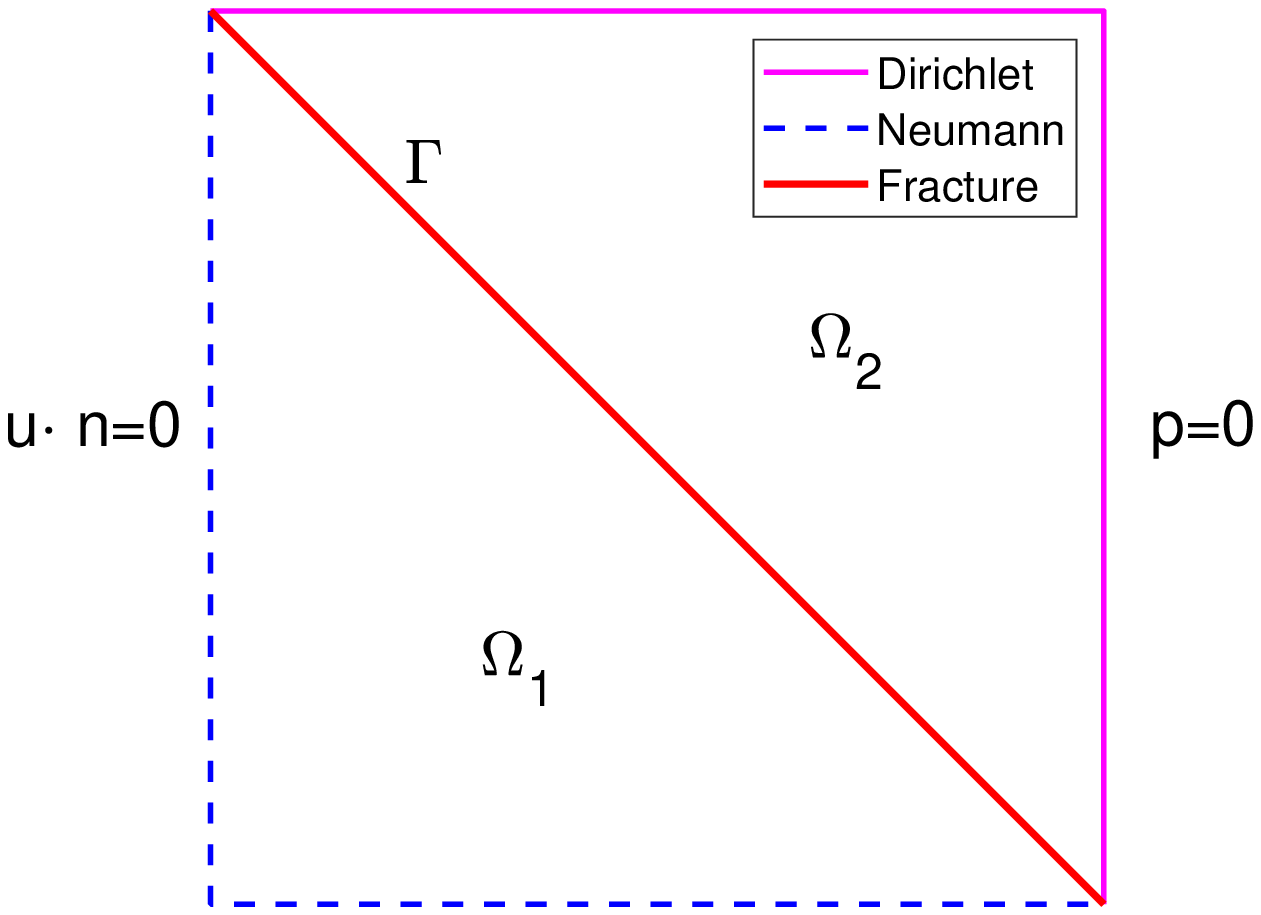}
    \includegraphics[width=0.45\textwidth]{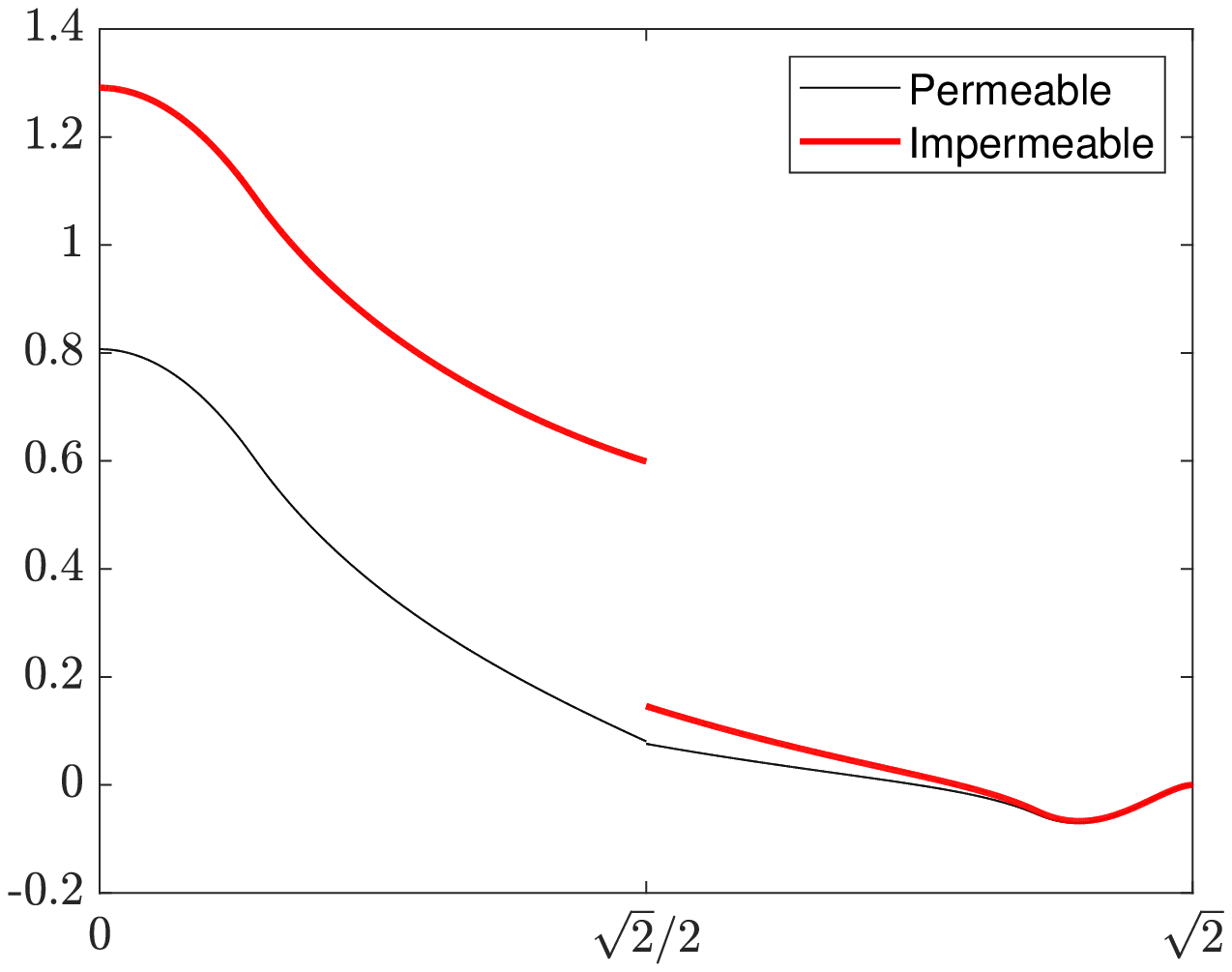}
    \caption{Domain configuration (left) and pressure profile along $x=y$ for Example~\ref{subsec:qfs}.\label{fig:qfs_prob}}
\end{figure}
\begin{figure}[t]
    \centering
    \includegraphics[width=0.45\textwidth]{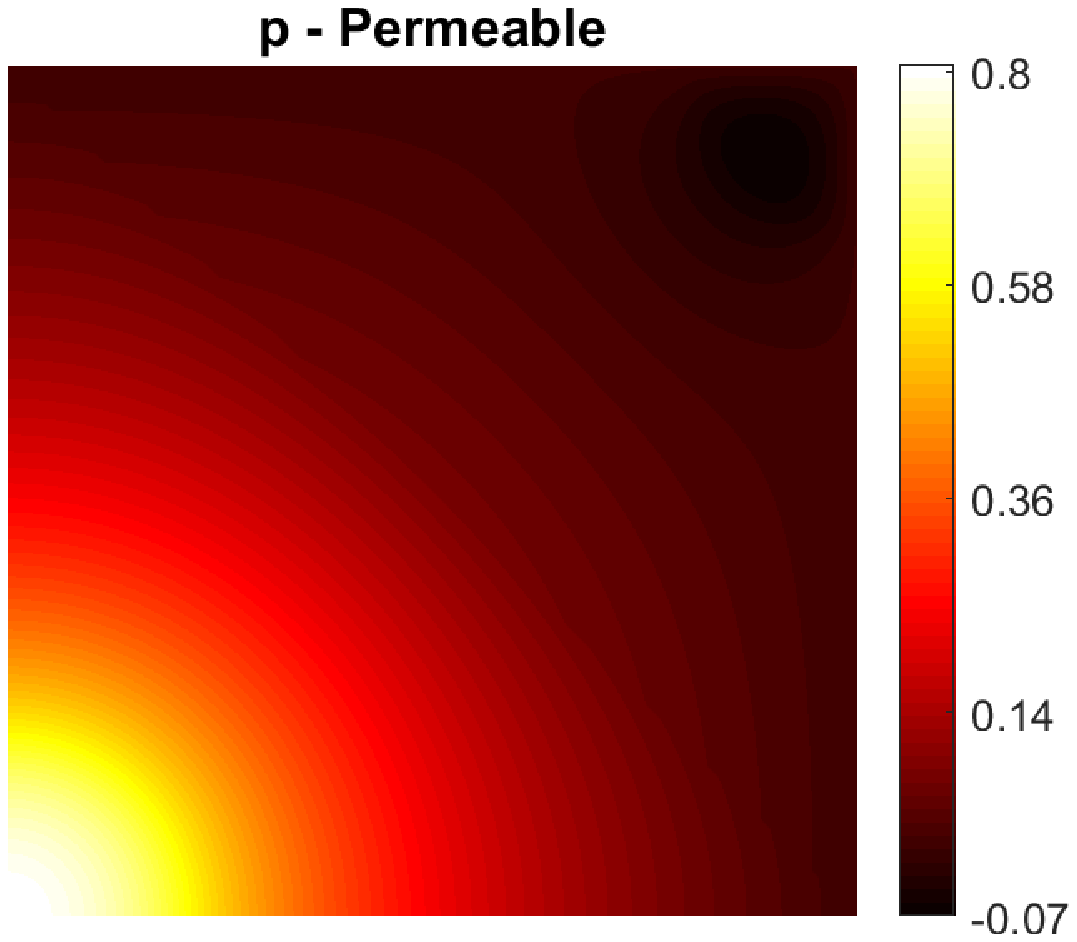}
    \includegraphics[width=0.45\textwidth]{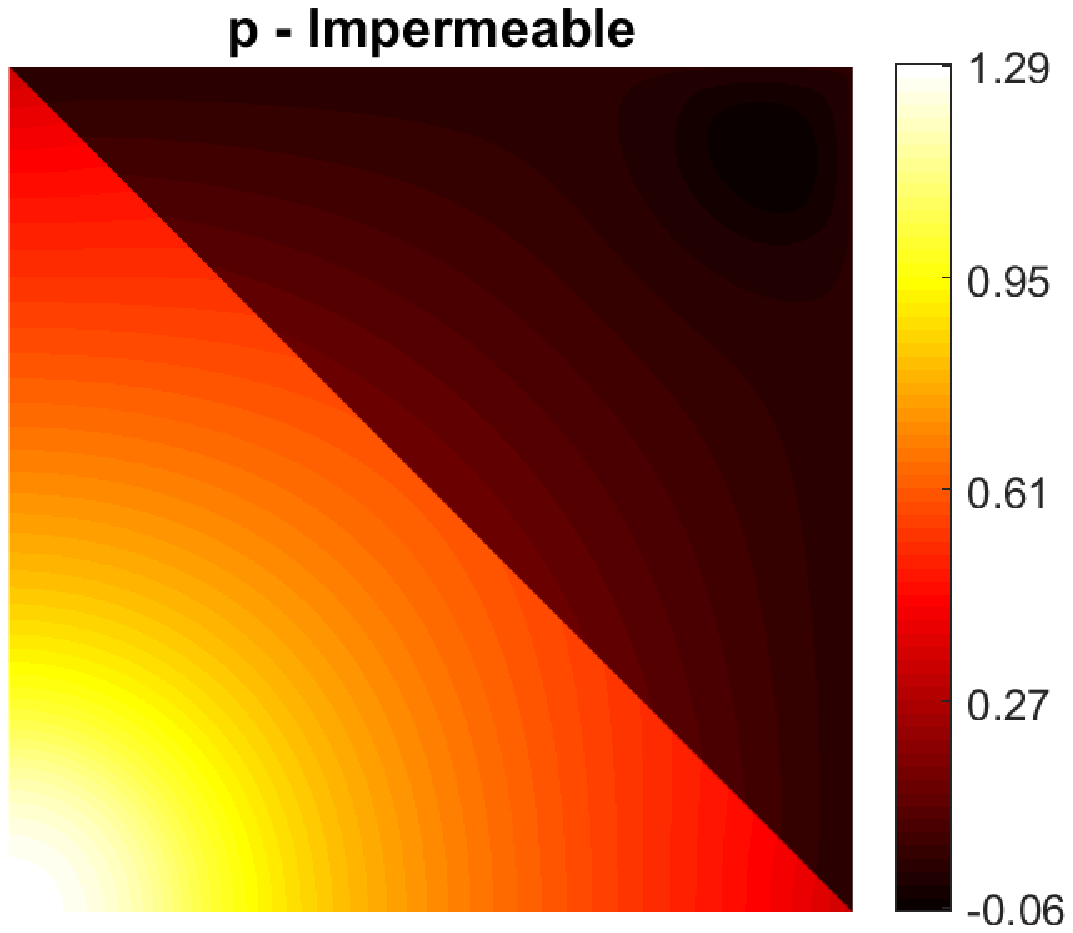}
    \caption{Pressure profile for Example~\ref{subsec:qfs} with permeable (left) and impermeable (right) fracture.\label{fig:qfs_p}}
\end{figure}
We conclude this section with a quarter five-spot problem.
A quarter five-spot problem emerges from petroleum engineering \cite{chen1998,petipjeans1999} and is frequently used to validate numerical algorithms \cite{Chave18, Antonietti19}.
Consider a unit square domain $\Omega$ with a diagonal fracture $\Gamma=\{(x,y):x+y=1\}$ with thickness $\ell_\Gamma=0.01$, see Figure~\ref{fig:qfs_prob}.
We set the boundary condition
\begin{equation*}
    \bm{u}\cdot\bm{n}=0\text{ on }\partial\Omega_1\backslash\Gamma,\quad p=0\text{ on }\partial\Omega_2\backslash\Gamma.
\end{equation*}
We model the injection and production by the source term
\begin{equation*}
    f = 10.1\Big(\tanh\left(200(0.2 - (x^2+y^2)^{\frac{1}{2}})\right)-\tanh\left(200(0.2-((x-1)^2+(y-1)^2)^{\frac{1}{2}})\right)\Big)
\end{equation*}
so that we have an injection well at $(0,0)$ and a production well at $(1,1)$.
Permeability for the bulk domain is chosen as $K=\bm{I}_{2\times 2}$.
As in \cite{Antonietti19}, we perform two numerical experiments: (1) Permeable fracture, $\kappa_\Gamma^n=1$ and $\kappa_\Gamma^*=100$; (2) Impermeable fracture, $\kappa_\Gamma^n=10^{-2}$ and $\kappa_\Gamma^*=1$.
The background mesh is chosen as a uniform rectangular mesh with $h\approx2^{-6}$ and we use cubic polynomials.
The bulk pressures are depicted in Figure~\ref{fig:qfs_p} and their profiles along the line $x=y$ are displayed in Figure~\ref{fig:qfs_prob}.

Both pressures with permeable and impermeable fractures have the largest value at the injection well $(0,0)$ and the smallest value at the projection well $(1,1)$.
The difference between the two pressure profiles are pronounced near the fracture.
Compared to the permeable fracture case, the impermeable fracture produces significant jump, see Figure~\ref{fig:qfs_prob} and \ref{fig:qfs_p}.
The profile produced with our method is qualitatively similar to that from \cite{Antonietti19}.

\section{Conclusion}\label{sec:conclusion}

In this paper we propose and analyze a staggered DG method combined with a standard conforming finite element method for the bulk and fracture allowing general polygonal elements even with arbitrarily small edges.
We impose the interface condition by replacing the average and jump terms with respect to the flux by the corresponding pressure term, which can guarantee the stability of the method.
The novel contributions of this paper are twofold.
First, convergence analysis allowing arbitrarily small edges is delivered, which sheds novel light on the analysis of staggered DG method for other physical problems.
Second, optimal flux $L^2$ error robust with respect to the heterogeneity and anisotropy of the permeability coefficients can be proved with the help of the Ritz projection. The numerical experiments presented indicate that our method is accurate, efficient and can handle anisotropic meshes without losing convergence order.
The proposed method has the flexibility of treating general meshes, which can be naturally adapted to solve problems on unfitted background grids.

\section*{Acknowledgements}

The research of Eric Chung is partially supported by the Hong Kong RGC General Research Fund (Project numbers 14304217 and 14302018), CUHK Faculty of Science Direct Grant 2018-19 and NSFC/RGC Joint Research Scheme (Project number HKUST620/15). The research of Eun-Jae Park is supported by
NRF-2015R1A5A1009350 and NRF-2019R1A2C2090021.


\end{document}